\documentclass[11pt,reqno]{amsart}

\usepackage[T1]{fontenc}
\usepackage[utf8]{inputenc}
\usepackage{lmodern}
\usepackage{microtype}
\usepackage{amsmath,amssymb,mathtools}
\usepackage{mathrsfs}
\usepackage{enumitem}
\usepackage[margin=1.12in]{geometry}
\usepackage{xcolor}
\usepackage[
  colorlinks=true,
  linkcolor=blue!55!black,
  citecolor=green!35!black,
  urlcolor=blue!60!black,
  pdfencoding=auto
]{hyperref}
\usepackage[nameinlink,capitalize,noabbrev]{cleveref}

\allowdisplaybreaks
\numberwithin{equation}{section}

\newtheorem{theorem}{Theorem}[section]
\newtheorem{proposition}[theorem]{Proposition}
\newtheorem{lemma}[theorem]{Lemma}
\newtheorem{corollary}[theorem]{Corollary}

\theoremstyle{definition}
\newtheorem{definition}[theorem]{Definition}
\theoremstyle{remark}
\newtheorem{remark}[theorem]{Remark}

\title[Convexity for Hessian equations]
{A Brunn--Minkowski inequality and Convexity for the 2-Hessian eigenvalue in convex domains}

\author{Jiahuan Li}
\address{Department of Mathematics, University of Science and Technology
of China, Hefei 230026, Anhui Province, China}
\email{jiahuan@mail.ustc.edu.cn}

\author{Xi-Nan Ma}
\address{Department of Mathematics, University of Science and Technology
of China, Hefei 230026, Anhui Province, China}
\email{xinan@ustc.edu.cn}

\author{Guohuan Qiu}
\address{State Key Laboratory of Mathematical Sciences, Academy of Mathematics and Systems Science, Chinese Academy of Sciences, Beijing 100190, China}
\address{Institute of Mathematics, Academy of Mathematics and Systems
Science, Chinese Academy of Sciences, No.\ 55 Zhongguancun East Road,
Beijing 100190, China}
\email{qiugh@amss.ac.cn}

\author{Paolo Salani}
\address{Dipartimento di Matematica e Informatica ``U.\ Dini'',
Universit\`a degli Studi di Firenze, 50134 Firenze, Italy}
\email{paolo.salani@unifi.it}

\hypersetup{
  pdftitle={Convexity for Hessian Equations: Brunn--Minkowski
  Inequalities for sigma2 and a Four-Dimensional sigma3 Counterexample},
  pdfauthor={Jiahuan Li, Xi-Nan Ma, Guohuan Qiu, Paolo Salani},
  pdfsubject={Convexity transforms, Hessian equations, and
  Brunn--Minkowski inequalities},
  pdfkeywords={Hessian equations, Brunn--Minkowski inequality,
  constant rank, nonconvex sublevel sets}
}

\begin{document}

\begin{abstract}
We prove the strict log-concavity of the positive first eigenfunction \(-u\) of the \(2\)-Hessian equation and the strict $1/2$-convexity of the solution for the corresponding torsion problem in smooth bounded uniformly convex
domains in $\mathbb{R}^{n}$. As applications, we establish
the associated Brunn--Minkowski inequalities. We also show that
this transformed-convexity phenomenon fails for \(3\)-Hessian equations by
constructing, in dimension four, a smooth uniformly convex domain whose
admissible zero-boundary solution has a nonconvex sublevel set.
\end{abstract}



\maketitle

\section{Introduction}\label{sec:introduction}

In 1976, in a legendary paper \cite{brascamp1976extensions}, Brascamp and Lieb established the log-concavity of the fundamental solution of diffusion equation with convex potential in bounded convex domain in $\mathbb{R}^{n}$. This in particular implies the log-concavity of the first Dirichlet eigenfunction of Laplace equation in convex domains, and also the Brunn--Minkowski inequality for the first eigenvalue, that is,
\begin{equation}\label{eq:bm-laplace}
\lambda\bigl((1-t)K_{0}+tK_{1}\bigr)^{-1/2}\geq (1-t)\lambda(K_{0})^{-1/2}+t\lambda(K_{1})^{-1/2},
\end{equation}
where $t\in[0,1]$ and $K_{0}$, $K_{1}$ are nonempty convex bodies in $\mathbb{R}^{n}$. In fact, this inequality holds for all compact connected domains having sufficiently regular boundary. As for the classical Brunn-Minkowski inequality, the equality case of \eqref{eq:bm-laplace} has its own interest, and Jerison \cite{jerison1996direct} pointed out that it is related to uniqueness of the solution for the Minkowski problem about $\lambda$. In \cite{colesanti2005brunn}, Colesanti provides a new proof of \eqref{eq:bm-laplace} for convex bodies, and proves that equality holds if and only if $K_{0}$ is homothetic to $K_{1}$. Simultaneously, Salani \cite{salani2005brunn} proved the Brunn-Minkowski inequality for the Dirichlet eigenvalue of the Monge-Amp\`ere operator and showed that equality again holds if and only if the involved convex sets are homothetic.

Laplacian and Monge-Amp\`ere operators are the extremal cases (corresponding to $k=1$ and $k=n$) of a class of operators known as {\em Hessian operators}: for $k=1,\ldots,n$, and a $C^{2}$ function $u$, the $k$-th Hessian operator $\sigma_{k}(D^{2}u)$ is the $k$-th elementary symmetric function of the eigenvalues of the Hessian matrix $D^2u$ of $u$. More explicitly:
$$
\sigma_k(D^2u)=\sum_{1\leq i_1<\dots<i_k\leq n}\lambda_{i_1}\cdots\lambda_{i_k}\,,
$$
where $\lambda_1,\dots,\lambda_n$ are the eigenvalues of $D^2u$. Notice that $\sigma_k(D^2u)$ can be also defined as the sum of all principal $k\times k$ minors of $D^2u$.
The operator $\sigma_k$, for $k>1$, is fully nonlinear and it is elliptic only when restricted to a suitable class of admissible functions, called {\em $k$-convex functions}:
a function $u\in C^2(\Omega)$ is said {\em $k$-convex} in an open set $\Omega\subset\mathbb{R}^n$ if $\sigma_j(D^2u)\geq 0$ in $\Omega$ for $j=1,\dots,k$.
Notice that $n$-convexity is equivalent to usual convexity in the class of $C^2$ functions, and it implies $k$-convexity for any $k=1,\dots,n$. Similarly to usual convexity, $k$-convexity, for $k=1,\dots, n-1$, can be defined also for $C^2$ sets in $\mathbb{R}^n$ using the symmetric functions of the principal curvatures. For sets, $(n-1)$-convexity coincides with usual convexity. Noticeably, every regular level set of a $k$-convex function is $(k-1)$-convex, with the natural orientation.

Hessian operators give rise to Hessian equations. Following the seminal work of Caffarelli, Nirenberg, and Spruck \cite{caffarelli1985dirichlet}, Dirichlet problems for Hessian equations have been extensively studied. On the geometric side, Guan and Ma \cite{guan2003christoffel} established a full-rank theorem and an existence result for strictly convex solutions of the Christoffel--Minkowski Hessian equation on the sphere. Caffarelli, Guan, and Ma \cite{caffarelli2007constant} subsequently developed a constant-rank theorem for a broad class of fully nonlinear elliptic equations, providing an important tool for the study of convexity. For Neumann boundary conditions, Ma and Qiu \cite{ma2019neumann} derived global a priori estimates up to second order and proved the existence of classical solutions in smooth uniformly convex domains, thereby giving an affirmative answer to a conjecture of Trudinger.
In particular, given $k\in\{2,\dots,n-1\}$ and a smooth uniformly $(k-1)$-convex domain $\Omega\subset\mathbb{R}^n$, the eigenvalue problems for the Laplacian and Monge-Amp\'ere operators are generalized as follows:
\begin{equation}\label{eq:eigenvalue-problem}
\begin{cases}
\sigma_{k}(D^{2}u)=\lambda_k(\Omega)(-u)^{k},\quad u<0 & \text{in }\Omega,\\
u=0 & \text{on }\partial \Omega\,.
\end{cases}
\end{equation}
Wang \cite{wang1994class} proved that, for $1<k<n$, there exists a unique positive eigenvalue $\lambda_k(\Omega)$ for which this problem is solvable, and that the corresponding negative $k$-convex solution $u\in C^{\infty}(\Omega)\cap C^{1,1}(\overline{\Omega})$ is unique up to multiplication by a positive factor.
Equivalently, one may define
\begin{equation}\label{eq:lambda-def}
\lambda_k(\Omega)=
\inf_{0\neq \varphi\in\Phi_0^k(\Omega)}
\frac{-\displaystyle\int_{\Omega}\varphi\,\sigma_{k}(D^{2}\varphi)\,dx}
{\displaystyle\int_{\Omega}|\varphi|^{k+1}\,dx},
\end{equation}
where
\[
\Phi_0^k(\Omega)=\left\{\varphi\in C^{2}(\Omega)\cap C(\overline{\Omega}):
\varphi=0\ \text{on }\partial\Omega,
\ \varphi\ \text{is $k$-convex in }\Omega\right\}.
\]
Clearly, this functional $\lambda_k(\Omega)$ is homogeneous of order $-2k$.

As we said, Laplace equation and Monge-Amp\'ere equation are just particular cases of $k$-Hessian equations, so it is natural to ask if results similar to the ones described at the beginning hold for all the $k$-Hessian equations. In particular, the questions are whether for $n\geq 3$ and $k=2,\dots,(n-1)$ the eigenvalue $\lambda_k$ satisfies a Brunn-Minkowski inequality, and whether $-u$ is log-concave (equivalently, whether $-\log(-u)$ is convex) or shares any other convexity property when $\Omega$ is convex.

A positive answer to these two questions has been given in \cite{ma2008convexity, liumaxu2010cc} and later in \cite{salani2012cc}, but only for $k=2$ and $n=3$. Here we treat the case $k=2$ for $n>3$, proving the following two theorems (for simplicity, we write $\lambda$ in place of $\lambda_2$ in the rest of the paper).

\begin{theorem}\label{thm:strict-convex}
Let $n\geq4$. Assume $K$ is a bounded uniformly convex smooth domain in $\mathbb{R}^{n}$ and let $u\in C^{\infty}(K)\cap C^{1,1}(\overline{K})$ be the unique (up to multiplication by a positive factor) admissible solution of
\begin{equation}\label{eq:s2-eigen}
\begin{cases}
\sigma_{2}(D^{2}u)=\lambda(K)(-u)^{2},\quad u<0 & \text{in }K,\\
u=0 & \text{on }\partial K\,.
\end{cases}
\end{equation}
Then $v=-\log(-u)$ is strictly convex; more precisely,
\[
D^2v>0\qquad\text{in }K.
\]
\end{theorem}

\begin{theorem}\label{thm:bm}
Let $n\geq4$, and $t\in[0,1]$. Suppose $K_{0},K_{1}$ are bounded uniformly convex smooth domains in $\mathbb{R}^{n}$. Then the functional $\lambda$ satisfies the inequality
\begin{equation}\label{eq:bm-s2}
\lambda\bigl((1-t)K_{0}+tK_{1}\bigr)^{-1/4}\geq (1-t)\lambda(K_{0})^{-1/4}+t\lambda(K_{1})^{-1/4}.
\end{equation}
Moreover, if $0<t<1$, equality holds if and only if $K_{0}$ and $K_1$ are homothetic.
\end{theorem}


Our method follows the lines of \cite{colesanti2005brunn, ma2008convexity, liumaxu2010cc, salani2012cc}, and it is based on a crucial algebraic property, see Theorem \ref{thm:main}, which establishes the convexity of a functional on the cone of positive definite symmetric matrices.
\medskip

We also obtain the following theorem, which generalizes to $n\geq 4$ the corresponding result from  \cite{ma2008convexity, liumaxu2010cc, salani2012cc}.
\begin{theorem}\label{thm:mainssss}
Suppose $u\in C^{\infty}(\overline K)$ is the admissible solution of \begin{equation}\label{eq:main-s2}
\sigma_{2}(D^{2}u)=1\quad\text{in }K
\qquad
u=0\quad\text{on }\partial K\,,
\end{equation}
where $K$ is a bounded uniformly convex smooth domain in $\mathbb{R}^{n}$, $n\geq4$. Then
\[
v:=-(-u)^{1/2}
\]
is strictly convex, and the power $1/2$ is sharp.
\end{theorem}

Of course, we can also obtain a Brunn-Minkowski inequality for the corresponding {\em $2$-torsional rigidity}
$$
\tau_2(\Omega)=\left(\int_\Omega u\,dx\right)^2\,,
$$
where $u$ is the solution of \eqref{eq:main-s2}.
\begin{theorem}\label{thm:bmtau}
Let $n\geq4$, and $t\in[0,1]$. Suppose $K_{0},K_{1}$ are bounded uniformly convex smooth domains in $\mathbb{R}^{n}$. Then
\begin{equation}\label{eq:bm-tau}
\tau_2\bigl((1-t)K_{0}+tK_{1}\bigr)^{1/(2n+4)}\geq (1-t)\tau_2(K_{0})^{1/(2n+4)}+t\tau_2(K_{1})^{1/(2n+4)}.
\end{equation}
Moreover, if $0<t<1$, equality holds if and only if $K_{0}$ and $K_1$ are homothetic.
\end{theorem}

We finally notice that, from \eqref{eq:bm-s2} and \eqref{eq:bm-tau}, with a standard argument (see \cite{BianchiniSalani}), we can eventually obtain Urysohn-type inequalities for $\lambda$ and $\tau_2$
\begin{theorem}\label{urysohn}
Among bounded smooth uniformly convex domains with given mean width, $\lambda$ is minimized by balls, while $\tau_2$ is maximized by balls.
\end{theorem}

Our next result gives a negative answer already for
\((k,n)=(3,4)\), the lowest-dimensional case in which \(\sigma_3\) is
an intermediate Hessian operator rather than the Monge--Amp\`ere
operator.

\begin{theorem}
\label{thm:main-E}
One can find a smooth bounded
uniformly convex domain
\(\Omega\subset\mathbb R^4\) for which the unique smooth
admissible solution \(u\) of
\begin{equation}\label{eq:main-E}
  \begin{cases}
    \sigma_3(D^2u)=1&\text{in }\Omega,\\
    u=0&\text{on }\partial\Omega,\\
    D^2u\in\Gamma_3&\text{in }\Omega
  \end{cases}
\end{equation}
has a nonconvex sublevel region.  More precisely,
\begin{equation}\label{eq:main-E-fixed-level}
  \left\{
    q\in\Omega:
    u(q)<-\frac{3}{125}
  \right\}
\end{equation}
is not convex.
\end{theorem}

The obstruction in Theorem~\ref{thm:main-E} is stronger than the
failure of the square-root transform.

\begin{corollary}
\label{thm:main-F}
For every solution \(u\) furnished by
Theorem~\ref{thm:main-E}, if \(\Psi\) is strictly increasing on the
range of \(u\), then \(\Psi\circ u\) is not
convex.  In particular,
\[
  -(-u)^\alpha
\]
is not convex for any \(\alpha>0\).
\end{corollary}

The rest of the paper is organized as follows. In Section~2 we recall the basic facts about hyperbolic polynomials that will be used throughout the paper. In Section~3 we prove the matrix convexity property underlying the \(2\)-Hessian argument. Section~4 is devoted to the constant-rank theorem for the logarithmic transform, and in Section~5 we combine this theorem with a convex-envelope argument to prove Theorem~\ref{thm:strict-convex}. In Section~6 we prove Theorem~\ref{thm:bm}, Theorem~\ref{thm:mainssss} and derive the corresponding results for the \(2\)-Hessian torsion problem, including the Brunn--Minkowski and Urysohn-type inequalities. Finally, in Section~7 we construct the four-dimensional \(3\)-Hessian counterexample and prove Theorem~\ref{thm:main-E} and Corollary~\ref{thm:main-F}.
\section{Preliminaries}

We first introduce some notation, then we will recall some standard results about hyperbolic polynomials which will be needed later.
\medskip

Given a natural number $n>1$, we denote by $S^n$ the space of $n\times n$ real symmetric matrices; $S^n_+\subset S^n$ is the (closed) cone of positive semidefinite
matrices, while $S^n_{++}\subset S^n_+$ is the (open) cone of positive definite matrices. By $I$ we denote the $n\times n$ identity matrix.

For $A\in S^n$ and $k\in\{0,\dots,n\}$, we denote by $\sigma_k(A)$ the $k$-th symmetric function of the eigenvalues of $A$, that is,
$$
\sigma_0(A)=1,\qquad\sigma_1(A)=\sum_{i=1}^n\lambda_i\,,\qquad \sigma_k(A)=\sum_{1\leq i_1<\dots<i_k\leq n}\lambda_{i_1}\cdots\lambda_{i_k}\quad\text{for }k\geq 2\,,
$$
where $\lambda_1\leq\lambda_2\leq\dots\leq\lambda_n$ are the eigenvalues of $A$.
Notice that $\sigma_1(A)=\text{Tr}(A)$, the trace of $A$, while $\sigma_n(A)=\det A$, the determinant of $A$. The operator $\sigma_k$ can equivalently be defined by the formula
\begin{equation}\label{detexpansion}
\det(A+tI)=\sum_{k=0}^n\sigma_k(A)t^{n-k}\,.
\end{equation}

If $A\in S^n$ is invertible, i.e., $\det(A)\neq 0$, we denote by $\operatorname{adj}(A)$ the adjugate (or classical adjoint) matrix of $A$, that coincides with the transpose of the cofactor matrix of $A$. In the sequel, we will make a decisive use of the following well known formula:
\begin{equation}\label{adjA}
A^{-1}=\frac{\operatorname{adj} A}{\det A}\,.
\end{equation}

\subsection{Hyperbolic polynomials}
For this section, we mainly refer to  the work of G{\aa}rding \cite{garding1959inequality}, and also to \cite{guler1997,bauschke2001hyperbolic,renegar2006hyperbolic,branden2014}.
\begin{definition}[Hyperbolic Polynomials and Hyperbolicity Cones]
Let $p$ be a homogeneous real polynomial on a real vector space $V$. If there exists a direction $e\in V$ such that $p(e)>0$ and, for every $x\in V$, the one-variable polynomial
\[
t\mapsto p(x+te)
\]
has only real roots, then $p$ is said {\em hyperbolic with respect to $e$}.

The connected component of $\{x:p(x)>0\}$ containing $e$ is called the {\em hyperbolicity cone of $p$} and is denoted by $\Gamma_{p}$. Its closure is denoted by $\overline\Gamma_p$.
\end{definition}
If $p$ is hyperbolic with respect to $e$, then it is hyperbolic with respect to every $v\in\Gamma_p$, and the hyperbolicity cone of $p$ with respect to $v$ coincides with $\Gamma_p$, so it does not depend on the direction and it is consistent not to refer to the direction when considering the hyperbolicity cone.

A cornerstone result by G{\aa}rding is that hyperbolicity cones are convex.
\medskip

We recall the standard fact that directional derivatives preserve hyperbolicity, so differentiation is an easy way to produce hyperbolic polynomials.
\begin{theorem}[\cite{branden2014} Lemma 4 ]\label{thm:directional-hyperbolic}
Let $p$ be hyperbolic with respect to $e$, with hyperbolicity cone $\Gamma_{p}$. If $v\in \overline\Gamma_{p}$, then the directional derivative $D_{v}p$ of $p$ in direction $v$ is hyperbolic with respect to $e$ as well, and its hyperbolicity cone contains $\Gamma_{p}$.
\end{theorem}
This follows from Rolle's theorem: along every line one obtains a real-rooted one-variable polynomial, and its derivative is again real-rooted. Directions in the closed cone are obtained by approximation from interior directions.
\medskip

The most important example in this context is the determinant $p(A)=\det A$,  which is a homogeneous polynomial (of degree $n$) on $S^n$ (identified with $\mathbb{R}^{n(n+1)/2}$): it is hyperbolic with respect to the identity matrix $I$, and its hyperbolicity cone is $S^{n}_{++}$.
Then, by repeated applications of Theorem \ref{thm:directional-hyperbolic} and  \eqref{detexpansion}, we obtain that $\sigma_k(A)$ is a hyperbolic polynomial of order $k$, for $k=1,\dots, n$, and $S^n_{++}=\Gamma_n\subset\Gamma_{n-1}\subset\dots\subset\Gamma_1$, where $\Gamma_k$ denotes the cone of hyperbolicity of $\sigma_k$.
\medskip

The following definition recalls another fundamental tool in the theory of hyperbolic polynomials.
\begin{definition}[Characteristic Roots]
Suppose $p$ is hyperbolic with respect to $e\in V$ and has degree $m$.
Then for every $x\in V$, we can write
\begin{equation}\label{deflambda}
p(x+te)=p(e)\prod_{i=1}^{m}\bigl(t+\lambda_i(x)\bigr),
\end{equation}
and assume without loss of generality that
\[
\lambda_1(x)\ge \lambda_2(x)\ge \cdots \ge \lambda_m(x).
\]
The number $\lambda_i(x)$, for $i=1,\dots, m$, is called  the $i$-th largest
\emph{characteristic root} of $x$ (with respect to $p$ and $e$). In other words, the characteristic roots $\lambda_1(x),\dots,\lambda_m(x)$ are the roots of the polynomial $t\to p(x-te)$.
The corresponding map
\[
\lambda:V\longrightarrow \mathbb{R}^m,
\qquad
x\longmapsto
\bigl(\lambda_1(x),\ldots,\lambda_m(x)\bigr),
\]
is called the \emph{characteristic map} (with respect to $p$ and $e$).
\end{definition}
Notice that, in particular, it holds
$$
p(x)=p(e)\prod_{i=1}^m\lambda_i(x)\,.
$$
G{\aa}rding's \cite{garding1959inequality} equation (2) tells
\begin{equation}\label{lambdagarding}
\lambda_i(rx+se)=r\lambda_i(x)+s\quad\text{if }r\geq0\,,\quad i=1,\dots, m\,.
\end{equation}
Hence the characteristic roots are positively homogeneous and continuous.

The hyperbolicity cone can be characterized using the characteristic roots as follows
\begin{equation}\label{Gammaplambda}
\Gamma_p=\{x\in V\,:\, \lambda_m(x)>0\}\,,
\end{equation}
see for instance Fact 2.7 in \cite{bauschke2001hyperbolic}.

G{\aa}rding also showed that $\lambda_m$ is concave, or equivalently that $\lambda_1$ is convex, since $\lambda_1(-x)=-\lambda_m(x)$ (see \cite{bauschke2001hyperbolic}). An important extension of this property is given in  \cite{bauschke2001hyperbolic}.
\begin{proposition}[\cite{bauschke2001hyperbolic} Theorem 3.9 ]\label{flambdathm}
If $f: \mathbb{R}^m\to (-\infty,+\infty]$ is convex and symmetric, then
 $f\circ\lambda$ is convex.
\end{proposition}
Here, by "symmetric" we mean $f(\lambda_1,\dots,\lambda_m)=f(\lambda_{\pi(1)},\dots,\lambda_{\pi(m)})$ for every permutation $\pi$ of $\{1,\dots,m\}$.

In fact the above theorem is a corollary of a stronger property.
\begin{proposition}[\cite{bauschke2001hyperbolic} Lemma 3.8 ]\label{flambdaprop}
If $f: \mathbb{R}^m\to (-\infty,+\infty]$ is convex and symmetric, then
for $x,y\in V$ and $\alpha\in(0,1)$, it holds
$$
f(\lambda((1-\alpha)x+\alpha y))\leq f((1-\alpha)\lambda(x)+\alpha\lambda(y))\,
$$
Moreover, if $f$ is strictly convex, the inequality is strict unless $\lambda((1-\alpha)x+\alpha y)=(1-\alpha)\lambda(x)+\alpha\lambda(y)$.
\end{proposition}

An easy consequence of Proposition \ref{flambdaprop} is the following property, whose proof is given for the completeness.
\begin{proposition}\label{thm:main0}
Let $p$ be a homogeneous hyperbolic polynomial of degree $m$, with hyperbolicity cone $\Gamma_{p}$. If $e\in\Gamma_{p}$, then
\[
\Phi(x)=\frac{D_{e}p(x)}{p(x)}
\]
is convex in $\Gamma_p$.

Moreover,
\begin{equation}\label{strictPhi}
\Phi((1-\alpha)x+\alpha y)<(1-\alpha)\Phi(x)+\alpha\,\Phi(y)
\end{equation}
unless
\[
\lambda((1-\alpha)x+\alpha y)
=(1-\alpha)\lambda(x)+\alpha\lambda(y).
\]
\end{proposition}
\begin{proof}
First, let us consider the case $e\in\Gamma_p$.

Differentiating \eqref{deflambda} at $t=0$ yields
$$
D_ep(x)=p(e)\sum_{i=1}^m\prod_{j\neq i}\lambda_j(x)\,.
$$
On the other hand
$$
p(x)=p(e)\prod_{i=1}^m\lambda_i(x)\,,
$$
hence
\begin{equation}\label{flambda}
\Phi(x)=\frac{D_ep(x)}{p(x)}=\sum_{i=1}^m\frac{1}{\lambda_i(x)}=f(\lambda(x))\,,
\end{equation}
where
$$
f(\lambda)=\sum_{i=1}^m\frac1{\lambda_i}=\frac{\sigma_{m-1}(\lambda)}{\sigma_m(\lambda)}\,.
$$
Extend $f$ to all of $\mathbb R^m$ by setting $f(\lambda)=+\infty$ outside
$\mathbb R^m_{++}$.  This extended-valued function is convex and symmetric,
and its restriction to $\mathbb R^m_{++}$ is strictly convex.  Since
\eqref{Gammaplambda} implies that $\lambda(x)\in\mathbb R^m_{++}$ for
$x\in\Gamma_p$, Theorem \ref{flambdathm} gives the convexity of $\Phi$.
Furthermore, Proposition \ref{flambdaprop} and the strict convexity of the
restriction of $f$ to $\mathbb R^m_{++}$ yield \eqref{strictPhi}.
\end{proof}
Regarding convexity, we may in fact prove a stronger result, precisely the following.
\begin{proposition}\label{concavity}
In the same assumptions of Proposition \ref{thm:main0}, the function
$$
\Psi(x)=\Phi(x)^{-1}=\frac{p(x)}{D_{e}p(x)}\quad\text{ is concave in }\Gamma_p\,.
$$
\end{proposition}
\begin{proof}
By G{\aa}rding's quotient-concavity theorem, the symmetric function
\[
q(\lambda)=\frac{\sigma_m(\lambda)}{\sigma_{m-1}(\lambda)}
=\left(\sum_{i=1}^m\frac1{\lambda_i}\right)^{-1}
\]
is concave on $\mathbb R^m_{++}$. Equivalently, $-q$ is convex there. Since
\[
-\Psi(x)=-q(\lambda(x)),
\]
the characteristic-root convexity theorem, applied to the convex symmetric function $-q$ (with its extended-valued convex extension outside $\mathbb R^m_{++}$), shows that $-\Psi$ is convex on $\Gamma_p$. Hence $\Psi$ is concave.
\end{proof}
Since the concavity of a nonnegative function $g$ easily implies the convexity of $1/g$, but not viceversa, Proposition \ref{concavity} implies the first part of Proposition \ref{thm:main0} (but not viceversa). Notice however that
$$
g(\lambda)=-f(\lambda)^{-1}=-\left(\sum_{i=1}^m\frac1{\lambda_i}\right)^{-1}
$$
is not strictly convex, so we can not say anything more than concavity of $\Psi$.
\begin{remark}
The convexity conclusion for $D_ep/p$ remains valid for $e\in\overline\Gamma_p$. Indeed, fix $d\in\Gamma_p$ and put $e_\varepsilon=e+\varepsilon d\in\Gamma_p$. Then $D_{e_\varepsilon}p/p$ is convex on $\Gamma_p$ and
\[
\frac{D_{e_\varepsilon}p(x)}{p(x)}
=
\frac{D_ep(x)}{p(x)}
+
\varepsilon\frac{D_dp(x)}{p(x)}
\longrightarrow
\frac{D_ep(x)}{p(x)}.
\]
Thus $D_ep/p$ is the pointwise limit of convex functions and is convex. Here the characteristic roots continue to be taken with respect to a fixed interior hyperbolicity direction; no characteristic-root formula relative to a boundary direction is being asserted.
\end{remark}

If $p$ is of a special type, it is possible to improve Proposition \ref{flambdaprop}. For this, let us recall the following definition form \cite{bauschke2001hyperbolic}.
\begin{definition}\label{def:complete}
The hyperbolic polynomial $p$ is called {\em complete} if
\[
\{x\in V:\lambda(x)=0\}=\{0\}.
\]
\end{definition}

\cite{bauschke2001hyperbolic} Fact 2.9 gives the equivalent descriptions
\[
\{x:\lambda(x)=0\}
=
\{x:x+\Gamma_p=\Gamma_p\}
=
\{x:p(tx+y)=p(y),\ \forall y\in V,\ \forall t\in\mathbb{R}\}.
\]
Therefore, for a hyperbolic polynomial $p$,
\[
p\ \text{is complete}
\quad\Longleftrightarrow\quad
\text{there is no nonzero }L\text{ such that }p(y+tL)=p(y)
\]
for all $y$ and $t$. Since
\[
D_{L}p\equiv0
\quad\Longleftrightarrow\quad
p(y+tL)=p(y),\qquad \forall y,\ t,
\]
we will use the following equivalent form:
\begin{equation}\label{eq:complete-equivalent}
p\ \text{is complete}
\quad\Longleftrightarrow\quad
L_p:=\{L:D_{L}p\equiv0\}=\{0\}.
\end{equation}
If $p$ is complete and $f$ is strictly convex, Proposition \ref{flambdaprop} can be improved and $f\circ\lambda$ results strictly convex, see Theorem 4.5 in \cite{bauschke2001hyperbolic}. In particular, what we will use is the following consequence.
\begin{theorem}[\cite{bauschke2001hyperbolic} Corollary 4.7]\label{thm:bgls}
Let $p$ be hyperbolic with hyperbolicity cone $\Gamma_p$. For $e\in \Gamma_p$, define
\[
F(x)=-\log p(x),\qquad h_{e}(x)=-(\nabla F(x))(e)=\frac{D_{e}p(x)}{p(x)},
\]
Then $F$ and $h_{e}$ are convex on $\Gamma_p$.

Moreover, if $p$ is complete, then $F$ and $h_{e}$ are strictly convex on $\Gamma_p$.
\end{theorem}

\section{A crucial convexity property}

The aim of this section is to establish the following property, which is crucial for our method.
\begin{theorem}\label{thm:main}
Let $n\geq2$, and fix $\alpha\in\mathbb{R}^{n}$ with $|\alpha|=1$. Set
\[
P=I-\alpha\alpha^{T}.
\]
Then the functional $\Phi:S^n_{++}\to\mathbb{R}$, defined as
\begin{equation}\label{eq:main-ineq}
\Phi(A)=\frac{\sigma_{2}(A^{-1})}{\operatorname{Tr}(PA^{-1})}\,,
\end{equation}
is convex in $S^n_{++}$. If $n\geq 4$, $\Phi$ is strictly convex.
\end{theorem}
Notice that, since $(tA)^{-1}=t^{-1}A^{-1}$,
\[
\operatorname{Tr}(P(tA)^{-1})=t^{-1}\operatorname{Tr}(PA^{-1}),
\qquad
\sigma_{2}((tA)^{-1})=t^{-2}\sigma_{2}(A^{-1}).
\]
Thus $\Phi$ is positively homogeneous of degree $-1$ that is,
\[
\Phi(tA)=t^{-1}\Phi(A),\qquad t>0\,.
\]

The proof of Theorem \ref{thm:main} is based on Proposition \ref{thm:main0} and it will be given in details in the next Subsection.
Similarly, applying Proposition \ref{concavity} we can obtain the following.
\begin{proposition}\label{concavityPhi}
In the same assumptions and notation of Theorem \ref{thm:main}, the functional
$$
\Psi(A)=\Phi(A)^{-1}=\frac{\operatorname{Tr}(PA^{-1})}{\sigma_{2}(A^{-1})}
$$
is concave in $S^n_{++}$.
\end{proposition}
Similar observations to those following Proposition \ref{concavity} apply.
\medskip

We also recall the following property from the Appendix of \cite{alvarez1997convexity}.
\begin{proposition}\label{prop:matrix-convex}
Let $\alpha\in\mathbb{R}^{n}$ with $|\alpha|=1$, $P=I_{n}-\alpha\alpha^{T}$.Then the function
\[
f(A):=\frac{1}{\operatorname{Tr}(PA^{-1})}
\]
is concave in $A\in S^{n}_{++}$.
\end{proposition}

\subsubsection{Proof of Theorem \ref{thm:main}: convexity}

Now we can proceed to the proof of Theorem \ref{thm:main}. First we take care of the convexity statement and we divide it in a few steps.
\medskip

{\bf STEP 1}: {\em getting rid of the inverse matrix.}

For every $A\in S^{n}_{++}$ it holds
\begin{equation}\label{lem:inverse-remove}
\Phi(A)=\frac{\sigma_{n-2}(A)}{\operatorname{Tr}(P\,\operatorname{adj} A)}.
\end{equation}

Indeed, let the eigenvalues of $A$ be $\lambda_{1},\ldots,\lambda_{n}>0$. Then the eigenvalues of $A^{-1}$ are $\lambda_{1}^{-1},\ldots,\lambda_{n}^{-1}$.
Hence
\begin{equation}\label{sigma2inv}
\sigma_{2}(A^{-1})=\sum_{1\leq i<j\leq n}\frac{1}{\lambda_{i}\lambda_{j}}
=\frac{\sigma_{n-2}(A)}{\det A}.
\end{equation}
Formula \eqref{adjA} also gives
\[
\operatorname{Tr}(PA^{-1})=\frac{\operatorname{Tr}(P\,\operatorname{adj} A)}{\det A}.
\]
Dividing the two expressions yields the result.
\medskip

{\bf STEP 2}: {\em writing the numerator and denominator as derivatives of a hyperbolic polynomial}

Set
\[
b=\alpha\alpha^{T},\qquad P=I-b,
\]
and define
\[
R=b+\frac12P=\frac12(I+b).
\]
Notice that, in a suitable reference frame in which $\alpha=e_n$, we have
\[
b_{nn}=1,
\qquad
b_{ij}=0\quad\text{for }(i,j)\neq(n,n).
\]
So $R$ is positive definite, with eigenvalues $1,1/2,\ldots,1/2$.

Let
\[
g(A)=D_{P}\det(A)=\left.\frac{d}{dt}\det(A+tP)\right|_{t=0},
\]
that is the directional derivative of the determinant of $A$ in the direction $P$:

\begin{lemma}\label{lem:numerator}
For every $A\in S^{n}_{++}$,
\[
g(A)=\operatorname{Tr}(P\,\operatorname{adj} A).
\]
\end{lemma}

\begin{proof}
The first derivative formula for the determinant gives
\[
D_{P}\det(A)=\det A\,\operatorname{Tr}(A^{-1}P).
\]
Since $A^{-1}=(\operatorname{adj}A)/\det A$, this equals $\operatorname{Tr}(P\,\operatorname{adj}A)$.
\end{proof}
\begin{remark}
Notice that, by Theorem \ref{thm:directional-hyperbolic}, this tells that $g(A)$ is hyperbolic and $S^n_{++}\subseteq\Gamma_g$.
\end{remark}
\medskip

Next we show that the denominator is also a directional derivative

\begin{lemma}\label{lem:denominator}
For every $A\in S^{n}_{++}$,
\[
\sigma_{n-2}(A)=D_{R}g(A)\,.
\]
\end{lemma}

\begin{proof}
Since $g(A)=D_{P}\det(A)$,
\[
D_{R}g(A)=D_{R}D_{P}\det(A).
\]
Let $X=A^{-1}$. The second directional derivative formula for the determinant is
\[
D_{U}D_{V}\det(A)=\det(A)\cdot\left[
\operatorname{Tr}(XU)\operatorname{Tr}(XV)-\operatorname{Tr}(XUXV)
\right].
\]
Taking $U=P$ and $V=R$, we obtain
\[
D_{R}g(A)=\det (A)\left[
\operatorname{Tr}(XP)\operatorname{Tr}(XR)-\operatorname{Tr}(XPXR)
\right].
\]
Write
\[
s=\operatorname{Tr}X,\qquad c=\operatorname{Tr}(Xb)=\alpha^{T}X\alpha.
\]
Since $P=I-b$ and $R=\frac12(I+b)$,
\[
\operatorname{Tr}(XP)=s-c,\qquad
\operatorname{Tr}(XR)=\frac12(s+c).
\]
Therefore
\[
\operatorname{Tr}(XP)\operatorname{Tr}(XR)=\frac12(s-c)(s+c)
=\frac12(s^{2}-c^{2}).
\]
Moreover,
\[
\operatorname{Tr}(XPXR)=\frac12\operatorname{Tr}\bigl(X(I-b)X(I+b)\bigr).
\]
Expanding the matrix inside the trace gives
\[
X(I-b)X(I+b)=X^{2}+X^{2}b-XbX-XbXb.
\]
By cyclic invariance of the trace,
\[
\operatorname{Tr}(X^{2}b)=\operatorname{Tr}(XbX),
\]
so the middle two terms cancel after taking the trace. Since $b=\alpha\alpha^{T}$ is a rank-one projection,
\[
\operatorname{Tr}(XbXb)=(\alpha^{T}X\alpha)^{2}=c^{2}.
\]
Thus
\[
\operatorname{Tr}(XPXR)=\frac12\left(\operatorname{Tr}(X^{2})-c^{2}\right).
\]
It follows that
\[
\operatorname{Tr}(XP)\operatorname{Tr}(XR)-\operatorname{Tr}(XPXR)
=\frac12\left(s^{2}-\operatorname{Tr}(X^{2})\right)
=\sigma_{2}(X).
\]
Therefore
\[
D_{R}g(A)=\det (A)\,\sigma_{2}(A^{-1}).
\]
By \eqref{sigma2inv},
\[
\sigma_{2}(A^{-1})=\frac{\sigma_{n-2}(A)}{\det (A)},
\]
and hence $D_{R}g(A)=\sigma_{n-2}(A)$.
\end{proof}
\medskip

{\bf STEP 3}: {\em conclusion}.

Combining \eqref{lem:inverse-remove} with the previous lemmas, we obtain
\begin{equation}\label{eq:key-identification}
\Phi(A)=\frac{\sigma_{n-2}(A)}{\operatorname{Tr}(P\,\operatorname{adj} A)}
=\frac{D_{R}g(A)}{g(A)}.
\end{equation}
This identification is the key point of the proof.
\medskip

Since
\[
g(A)=D_{P}\det(A)
\]
and $P=I-\alpha\alpha^{T}\geq 0$, that is, $P$ is in the closure of the hyperbolicity cone of the determinant,
 Theorem \ref{thm:directional-hyperbolic} implies that
\[
g(A)=D_{P}\det(A)
\]
is also a hyperbolic polynomial, and that its hyperbolicity cone contains $S^{n}_{++}$.

On the other hand,
\[
R=\frac12(I+\alpha\alpha^{T})>0,
\]
so $R\in S^{n}_{++}$. Now, taking into account \eqref{eq:key-identification}, we can apply Proposition \ref{thm:main0} to $g$ in the direction $R$ to get the convexity of $\Phi$.

\subsubsection{Proof of Theorem \ref{thm:main}: strict convexity}
By Theorem \ref{thm:bgls}, to prove strict convexity it is sufficient to show that $p=g(A)$ is complete. We divide this part into steps too.
\medskip

{\bf STEP 1:} {\em the Block Formula}

To simplify slightly notation, let us denote $m=n-1$ in the rest of this section.

Through orthogonal transformation, we can assume $$P=\begin{bmatrix}I_{m}&0\\ 0&0\end{bmatrix}$$\,.

Given $A\in S^n_{++}$, let us write it as
\begin{equation}\label{A}
A=\begin{pmatrix}
C& u\\
u^{T} & c
\end{pmatrix}\,,
\end{equation}
that is, we denote by $C$ the matrix in $S^{m}_{++}$ with entries $c_{ij}=a_{ij}$ for $i,j=1,\dots,m$, $u=(a_{1n},a_{2n},\dots,a_{mn})^T$, $c=a_{nn}$.

Then
\begin{equation}\label{eq:g-block0}
g(A)=\operatorname{Tr}(P\,\operatorname{adj}A)
=c\,\sigma_{m-1}(C)-u^{T}\left[\sigma_{m-1}(C)C^{-1}-\det(C)C^{-2}\right]u\,.
\end{equation}
Indeed,
\[
g(A)=D_{P}\det(A)
=\left.\frac{d}{dt}
\det
\begin{pmatrix}
C+tI_{m} & u\\
u^{T} & c
\end{pmatrix}\right|_{t=0}.
\]
When $C+tI_{m}$ is invertible, the Schur complement formula gives
\[
\det
\begin{pmatrix}
C+tI_{m} & u\\
u^{T} & c
\end{pmatrix}
=\det(C+tI_{m})\cdot
\left(c-u^{T}(C+tI_{m})^{-1}u\right).
\]
Differentiating at $t=0$, the first part contributes
\[
c\,D_{I}\det(C)=c\,\operatorname{Tr}(\operatorname{adj}C)=c\,\sigma_{m-1}(C),
\]
and the second part gives
$$
\begin{array}{rl}
-u^{T}D_{I}\left(\det(C)C^{-1}\right)u&=-u^{T}\left[\sigma_{m-1}(C)C^{-1}+\det(C)\frac{d}{dt}(C^{-1})_{|t=0}\right]u\\
\\
&=-u^T\left[\sigma_{m-1}(C)C^{-1}-\det(C)C^{-2}\right]u\,.
\end{array}
$$
Putting the two contributions together, we obtain \eqref{eq:g-block0}.

Notice that $\sigma_{m-1}(C)C^{-1}-\det(C)C^{-2}=T_{m-2}(C)$, where $T_k$ denotes the so called {\em Newton Transform} of order $k$ of $C$, defined as
\[
T_{k}(C)=\sigma_{k}(C)I-\sigma_{k-1}(C)C+\sigma_{k-2}(C)C^{2}-\cdots+(-1)^{k}C^{k}.
\]
So, \eqref{eq:g-block0} can be rewritten also in the following more compact and elegant form
\begin{equation}\label{eq:g-block}
g(A)=c\,\sigma_{m-1}(C)-u^{T}T_{m-2}(C)u\,,
\end{equation}
and, since both sides are polynomial identities, in fact no invertibility assumption on $C$ is needed.
\medskip

{\bf STEP 2:} {\em analysis of Degeneracy Directions}

Assume $L\in L_{g}$, that is,
\[
D_{L}g(A)\equiv0.
\]
We want to prove that $L=0$.

Set
\[
L=
\begin{pmatrix}
M & v\\
v^{T} & \ell
\end{pmatrix}\,,
\]
where $M\in S^{m}$, $v$ is a vector of length $m$ and $\ell\in\mathbb{R}$.

First, take $u=0$ in \eqref{A}. Then, by \eqref{eq:g-block}, it holds
\[
g(A)=c\,\sigma_{m-1}(C).
\]
Thus
\[
D_{L}g(A)=\ell\,\sigma_{m-1}(C)+c\,D_{M}\sigma_{m-1}(C).
\]
This is identically zero for all $C$ and $c$, so
\begin{equation}\label{zero}
\ell=0,\qquad D_{M}\sigma_{m-1}(C)\equiv0.
\end{equation}
Take $C=I_m$, and let $\mu_1,\dots,\mu_m$ denote the eigenvalues of $M$. Since
\[
\sigma_{m-1}(I_m+tM)
=
\sum_{j=0}^{m-1}(m-j)\sigma_j(M)t^j,
\]
the second equation in \eqref{zero}, differentiated at $t=0$, gives
\[
(m-1)\operatorname{Tr}(M)=0,
\qquad\text{hence}\qquad
\operatorname{Tr}(M)=0.
\]
Now take $C=I_m+sN$, where $N\in S^m$ is arbitrary and $s$ is sufficiently small that $C>0$, and recall the expansion
\begin{equation}\label{sigmaNM}
\sigma_{m-1}(I_m+sN+tM)=\sum_{i=0}^{m-1}\sum_{j=0}^{m-1-i}\binom{m-(i+j)}{m-1-{i+j}}\sigma_{(i,j)}(N,M)s^it^j\,,
\end{equation}
where
$$
\sigma_{(i,j)}(N,M)=\frac1{i!j!}\frac{\partial^{i+j}\sigma_{i+j}(sN+tM)}{\partial s^i\partial t^j}|_{s=0,t=0}\,.
$$
In particular
$$
\sigma_{(0,1)}(N,M)=\operatorname{Tr}(M)\,,\qquad\sigma_{(1,1)}(N,M)=\operatorname{Tr}(M)\operatorname{Tr}(N)-\operatorname{Tr}(NM)\,.
$$
The second of \eqref{zero} now reads
\begin{equation}\label{DNDM}
\sum_{i=0}^{m-2}\binom{m-(i+1)}{m-1-(i+1)}\sigma_{(i,1)}(N,M)s^i=0\quad\text{for every sufficiently small }s\,.
\end{equation}
Taking the derivative with respect to $s$ at $s=0$ gives
\[
0=(m-2)\sigma_{(1,1)}(N,M)
=(m-2)\bigl(\operatorname{Tr}(M)\operatorname{Tr}(N)-\operatorname{Tr}(NM)\bigr).
\]
Here $m=n-1\geq3$, so $m-2>0$. Since $\operatorname{Tr}(M)=0$, we obtain
\begin{equation}\label{TrNM0}
\operatorname{Tr}(NM)=0\qquad\text{for every }N\in S^m.
\end{equation}
This implies $M=0$. For example, taking $N=ww^T$ gives $w^TMw=0$ for every $w\in\mathbb R^m$.

At this point, we already know $M=0$ and $\ell=0$. Returning to the full formula \eqref{eq:g-block0}, differentiating in direction $L$ means just differentiating with respect to $u$ in direction $v$, so we have
\[
D_{L}g(A)=-2v^T[\sigma_{m-1}(C)C^{-1}-\det(C)C^{-2}]u.
\]
This vanishes for all $C$ and $u$. Taking $C=I$, we have
\[
-2(m-1)v^{T}u=0,\qquad \forall u\in\mathbb{R}^{m}.
\]
Hence $v=0$. Therefore $L=0$, i.e.
\[
n\geq4\quad\Longrightarrow\quad L_{g}=\{0\}.
\]
Finally, by Theorem \ref{thm:bgls} we get the following corollary, which concludes the proof of Theorem \ref{thm:main}.
\begin{corollary}\label{cor:Phi-strict}
When $n\geq4$,
\[
\Phi(A)=\frac{D_{R}g(A)}{g(A)}
\]
is strictly convex on $S^{n}_{++}$.
\end{corollary}

The proof of Theorem \ref{thm:main} is complete.

\subsubsection{The Case $n=3$}
For the sake of completeness, we will explain why $n=3$ is a special case.

In this case $m=n-1=2$, so
\[
\sigma_{m-1}(C)=\sigma_{1}(C)=\operatorname{Tr}C\,,
\]
and formula \eqref{sigmaNM} reads simply
\[
\sigma_{m-1}(I+sN+tM)=2+s\operatorname{Tr}(N)+t\operatorname{Tr}(M)\,,
\]
whence
\[
D_ND_{M}\sigma_{1}(I)=0
\]
and we cannot obtain \eqref{TrNM0}, but only
\[
\ell=0,\qquad \operatorname{Tr}M=0.
\]
Returning to the full block formula, here $T_{m-2}=T_0=I_m$, so the identity $D_Lg\equiv0$ also gives $-2v^Tu=0$ for every $u\in\mathbb R^m$ and hence $v=0$. Thus
\[
L_{g}=
\left\{
\begin{pmatrix}
M & 0\\
0 & 0
\end{pmatrix}
: M=M^{T},\ \operatorname{Tr}M=0
\right\}.
\]
In coordinate-free form,
\[
L_{g}=\{L=L^{T}:L\alpha=0,\ \operatorname{Tr}L=0\}.
\]
Hence $g$ is not complete.




\section{An ad hoc Constant Rank Theorem}
Let $u$ be the admissible solution of problem \eqref{eq:s2-eigen}, normalized by $\|u\|_\infty=1$, and set
$v=-\log(-u)$. Then
\eqref{eq:s2-eigen} is equivalent to
\begin{equation}\label{eq:v-equation}
\begin{cases}
\sigma_{2}(D^{2}v)-\operatorname{Tr}(P(Dv)D^{2}v)=\lambda(K) & \text{in }K,\\
v(x)\to+\infty & x\to\partial K,
\end{cases}
\end{equation}
where $P(\nabla v)=|Dv|^{2}I-Dv\otimes Dv$ is the matrix with entries
\begin{equation}\label{eq:Pij}
P_{ij}=|\nabla v|^{2}\delta_{ij}-v_{i}v_{j},\quad i,j=1,\ldots,n.
\end{equation}
See Subsection \ref{subs} below for details. It follows that $P$ is positive semidefinite and
$\operatorname{Tr}(PD^{2}v)=\sum_{i,j=1}^{n}P_{ij}v_{ij}$. Moreover,
\[
D^2u=e^{-v}\bigl(D^2v-Dv\otimes Dv\bigr),
\]
so the admissibility of $u$ is equivalent to
$D^2v-Dv\otimes Dv\in\Gamma_2$. In particular, since $Dv\otimes Dv\geq0$, one also has $D^2v\in\Gamma_2$; however, the transformed admissibility condition needed below is the former one.

One of the main ingredients in the proof of Theorem \ref{thm:strict-convex} is the following constant-rank theorem, which states as follows (see e.g. \cite{caffarelli1985convexity} and \cite{korevaar1987convex}). Some related results have been obtained in \cite{ma2008convexity}.

\begin{theorem}\label{thm:revised}
Let $K\subset\mathbb{R}^{n}$ be a connected domain. Let
\[
v\in C^{4}(K),\qquad D^{2}v(x)\geq0,
\]
and assume that, for some constant $\lambda>0$, $v$ satisfies
\begin{equation}\label{eq:theorem-equation}
\sigma_{2}(D^{2}v)-\operatorname{Tr}(P(Dv)D^{2}v)=\lambda
\quad\text{in }K.
\end{equation}
Assume further that $v$ is admissible, namely
\begin{equation}\label{eq:admissible}
W(x):=D^{2}v(x)-Dv(x)\otimes Dv(x)\in\Gamma_{2}
\quad\text{for all }x\in K.
\end{equation}
Then
\[
\operatorname{rank}D^{2}v(x)
\]
is constant in $K$.
\end{theorem}

\subsection{A general Constant Rank Theorem}
Before giving the proof of Theorem \ref{thm:revised}, for convenience, we state a version of the Bian--Guan structural constant rank theorem used here. It is a simplified form of Theorem 1.2 in \cite{bian2010structural}, adapted to the present setting, which refines the original microscopic convexity principle in \cite{bian2009microscopic}:

\begin{theorem}[Bian--Guan level-set form]\label{thm:bg-level}
Let $\Omega\subset\mathbb{R}^{n}$ be a connected domain, and let $u\in C^{3,1}(\Omega)$ be a convex solution of
\[
F(D^{2}u,Du,u,x)=0,
\]
where
\[
F=F(r,p,z,x)\in C^{2,1}(S^{n}\times\mathbb{R}^{n}\times\mathbb{R}\times\Omega).
\]
Assume the following:
\begin{enumerate}
    \item[(1)] $F$ is elliptic along the solution, namely
    \[
    \bigl(F^{ij}(D^{2}u,Du,u,x)\bigr)>0,
    \qquad
    F^{ij}:=\frac{\partial F}{\partial r_{ij}}.
    \]
    \item[(2)] Along the solution, the nondegeneracy condition holds:
    \[
    F(0,Du(x),u(x),x)\neq0.
    \]
    \item[(3)] For each fixed relevant gradient $p$, the sublevel set
    \[
    \Gamma_{F}(p)
    :=
    \left\{
    (A,z,x)\in S^{n}_{++}\times\mathbb{R}\times\Omega:
    F(A^{-1},p,z,x)\leq0
    \right\}
    \]
    is locally convex near the relevant points.
\end{enumerate}
Then
\[
\operatorname{rank}D^{2}u(x)
\]
is constant in $\Omega$.
\end{theorem}
\subsection{Proof of Theorem \ref{thm:revised}}
We want to apply Theorem \ref{thm:bg-level} to get Theorem \ref{thm:revised}.

\subsubsection{Writing the Equation in a Form Suitable for the Constant Rank Theorem}\label{subs}
First, it is convenient to see in details how we pass from the equation for $u$ to the equations for $v$.

A straightforward calculation gives
$$
u=-e^{-v}\,,\quad Du=e^{-v}Dv\,,\quad D^2u=e^{-v}D^2v-e^{-v}Dv\otimes Dv\,.
$$
Starting from
\[
\sigma_2(D^2u)=\lambda(-u)^2
\]
and using the homogeneity of $\sigma_2$, we obtain
\begin{equation}\label{eqrewritten}
\sigma_2(D^2v-Dv\otimes Dv)-\lambda=0.
\end{equation}
We have the elementary identity
\begin{equation}\label{eq:s2-identity}
\sigma_{2}(r-p\otimes p)=\sigma_{2}(r)-\operatorname{Tr}(P(p)r)\,,
\end{equation}
where
\[
P(p)=|p|^{2}I-p\otimes p.
\]
Indeed, from
\[
\sigma_{2}(M)=\frac12\left((\operatorname{Tr}M)^{2}-\operatorname{Tr}(M^{2})\right),
\]
and since $p\otimes p$ has rank one and $\sigma_{2}(p\otimes p)=0$, we get
\[
\sigma_{2}(r-p\otimes p)
=\sigma_{2}(r)-|p|^{2}\operatorname{Tr}r+p^{T}rp.
\]
On the other hand,
\[
\operatorname{Tr}(P(p)r)
=\operatorname{Tr}\bigl((|p|^{2}I-p\otimes p)r\bigr)
=|p|^{2}\operatorname{Tr}r-p^{T}rp.
\]
Consequently,
if we define
\begin{equation}\label{eq:F-def-log}
F(r,p):=\sigma_{2}(r)-\operatorname{Tr}(P(p)r)-\lambda,
\end{equation}
we have
\[
F(r,p)=\sigma_{2}(r-p\otimes p)-\lambda,
\qquad r\in S^{n},\quad p\in\mathbb{R}^{n}.
\]
Accordingly, \eqref{eq:theorem-equation} reads
\begin{equation}\label{eq:F-zero-log}
F(D^{2}v,Dv)=0.
\end{equation}

Notice the function $F$ is independent of $z$ and $x$ (so these variables will not produce extra terms when Theorem \ref{thm:bg-level} is applied).

\subsubsection{Ellipticity}\label{Ellipticity}
The ellipticity of equation \eqref{eq:F-zero-log}, or equivalently of equation \eqref{eqrewritten}, is a straightforward consequence of the following easy proposition, which can be found for instance in \cite{salaniphd}.
\begin{proposition}[Proposizione 1.4.1 of \cite{salaniphd}]\label{prop:salaniphd}
Let $1\leq k\leq n$ and $A=(a_{ij})\in S^n$ with eigenvalues $\mu_1,\dots,\mu_n$. Then, the matrix
$$
A_k=(\sigma_k^{ij}(A))\,,\quad\text{where }\sigma_k^{ij}(A)=\frac{\partial\sigma_k(A)}{\partial a_{ij}}\,,
$$
has eigenvalues $\sigma_{k,1}(\mu_1,\dots,\mu_n),\dots,\sigma_{k,n}(\mu_1,\dots,\mu_n)$, where
$$
\sigma_{k,i}(\mu_1,\dots,\mu_n)=\frac{\partial\sigma_k(\mu_1,\dots,\mu_n)}{\partial\mu_i}=\sigma_{k-1}(\hat{\mu}_1,\dots,\hat{\mu}_n)\quad\text{with }\hat{\mu}_j=(1-\delta_{ij})\mu_j\,.
$$
\end{proposition}
\begin{proof}
For simplicity of notation, we use the Einstein convention of summation over repeated indices.

Let $C=(c_{ij})$ be the orthogonal matrix which diagonalizes $A$, that is,
$$\text{diag}(\mu_1,\dots,\mu_n)=C^TAC\,,
$$
or equivalently
$$
\mu_r=c_{ir}a_{ij}c_{rj}\,,\quad a_{ij}=c_{ir}\mu_rc_{rj}\,.
$$
With a straightforward calculation we obtain
$$
C^TA_kC=\text{diag}(\sigma_{k,1}(A),\dots,\sigma_{k,n}(A))\,.
$$
Indeed
$$
\sigma_{k,r}(A)=\frac{\partial\sigma_k(\mu_1,\dots,\mu_n)}{\partial\mu_r}=\frac{\partial\sigma_k(\mu_1,\dots,\mu_n)}{\partial a_{ij}}\frac{\partial a_{ij}}{\partial\mu_r}=\sigma_k^{ij}(A)c_{ir}c_{rj}\,,
$$
$$
\sigma_k^{ij}(A)=\frac{\partial\sigma_k(A)}{\partial a_{ij}}=\frac{\partial\sigma_k(A)}{\partial \mu_r}\frac{\partial\mu_r}{\partial a_{ij}}=\sigma_{k,r}(A)c_{ir}c_{rj}\,.
$$
\end{proof}
Now write
\[
W=D^2v-Dv\otimes Dv.
\]
Differentiating $F$ with respect to the Hessian variable gives
\begin{equation}\label{eq:Fij}
F^{ij}(r,p)
=
\frac{\partial\sigma_2(r-p\otimes p)}{\partial r_{ij}}.
\end{equation}
Thus, along the solution,
\[
F^{ij}(D^2v,Dv)=\sigma_2^{ij}(W)=T_1(W)_{ij},
\qquad
T_1(W)=\sigma_1(W)I-W.
\]
By the admissibility assumption \eqref{eq:admissible}, $W\in\Gamma_2$. Proposition \ref{prop:salaniphd} therefore implies that every eigenvalue of $T_1(W)$ is strictly positive. Hence
\[
\bigl(F^{ij}(D^2v,Dv)\bigr)=T_1(W)>0,
\]
which is precisely the ellipticity condition in Theorem \ref{thm:bg-level}.

\subsubsection{The Nondegeneracy Condition}\label{nondegeneracy}

For every $p\in\mathbb{R}^{n}$,
\[
F(0,p)=\sigma_{2}(0)-\operatorname{Tr}(P(p)0)-\lambda=-\lambda.
\]
Since $\lambda>0$,
\begin{equation}\label{eq:nondegenerate}
F(0,p)=-\lambda\neq0.
\end{equation}
Thus the nondegeneracy condition in Theorem \ref{thm:bg-level} is satisfied.

\subsubsection{Convexity of the Sublevel Set: The Case $p\neq0$}

Fix $p\neq0$. Set
\[
\alpha=\frac{p}{|p|},
\qquad
Q_{\alpha}=I-\alpha\otimes\alpha.
\]
Then
\[
P(p)=|p|^{2}Q_{\alpha}.
\]
Consider the sublevel set
\begin{equation}\label{eq:Gamma-p}
\Gamma_{p}:=\{A\in S^{n}_{++}:F(A^{-1},p)\leq0\}.
\end{equation}
By \eqref{eq:F-def-log}, the condition $F(A^{-1},p)\leq0$ is equivalent to
\[
\sigma_{2}(A^{-1})-|p|^{2}\operatorname{Tr}(Q_{\alpha}A^{-1})-\lambda\leq0.
\]
Since $A^{-1}>0$, $Q_{\alpha}\geq0$, and $Q_{\alpha}\neq0$,
\begin{equation}\label{eq:positive-denominator}
\operatorname{Tr}(Q_{\alpha}A^{-1})>0.
\end{equation}
Thus the previous inequality is equivalent to
\begin{equation}\label{eq:H-sublevel}
H_{\alpha,\lambda}(A)
:=
\frac{\sigma_{2}(A^{-1})-\lambda}
{\operatorname{Tr}(Q_{\alpha}A^{-1})}
\leq |p|^{2}.
\end{equation}

Combining Theorem \ref{thm:main} with Proposition \ref{prop:matrix-convex}, the function
\begin{equation}\label{eq:H-convex}
A\longmapsto
\frac{\sigma_{2}(A^{-1})}{\operatorname{Tr}(Q_{\alpha}A^{-1})}
-\frac{\lambda}{\operatorname{Tr}(Q_{\alpha}A^{-1})}
=
\frac{\sigma_{2}(A^{-1})-\lambda}{\operatorname{Tr}(Q_{\alpha}A^{-1})}
\end{equation}
is convex on $S^{n}_{++}$: the first term is convex by Theorem \ref{thm:main}, while the second is convex because $1/\operatorname{Tr}(Q_\alpha A^{-1})$ is concave. Therefore $\Gamma_{p}$ is a sublevel set of the convex function $H_{\alpha,\lambda}$ and is convex, hence locally convex.

\subsubsection{Convexity of the Sublevel Set: The Case $p=0$}

When $p=0$, the vector $\alpha=p/|p|$ is not defined.
We handle it by a limiting intersection argument.

Fix an arbitrary unit vector $\beta\in\mathbb{R}^{n}$ and set
\[
p_{\varepsilon}=\varepsilon\beta,\qquad \varepsilon>0.
\]
By the previous step, $\Gamma_{p_{\varepsilon}}$ is convex for every $\varepsilon>0$. We claim that
\begin{equation}\label{eq:intersection}
\Gamma_{0}=\bigcap_{\varepsilon>0}\Gamma_{p_{\varepsilon}}.
\end{equation}
Indeed, $A\in\Gamma_{p_{\varepsilon}}$ is equivalent to
\[
\sigma_{2}(A^{-1})-\lambda
\leq
\varepsilon^{2}\operatorname{Tr}(Q_{\beta}A^{-1}).
\]
If $A\in\bigcap_{\varepsilon>0}\Gamma_{p_{\varepsilon}}$, then this inequality holds for every $\varepsilon>0$. Letting $\varepsilon\to0$ yields
\[
\sigma_{2}(A^{-1})-\lambda\leq0,
\]
which is exactly $A\in\Gamma_{0}$.

Conversely, if $A\in\Gamma_{0}$, then
\[
\sigma_{2}(A^{-1})-\lambda\leq0.
\]
Since
\[
\varepsilon^{2}\operatorname{Tr}(Q_{\beta}A^{-1})\geq0,
\]
we immediately get
\[
\sigma_{2}(A^{-1})-\lambda
\leq
\varepsilon^{2}\operatorname{Tr}(Q_{\beta}A^{-1}),
\]
so $A\in\Gamma_{p_{\varepsilon}}$ for every $\varepsilon>0$. This proves \eqref{eq:intersection}.

The intersection of any family of convex sets is convex. Hence $\Gamma_{0}$ is convex. Therefore, for every $p\in\mathbb{R}^{n}$, the set
\[
\Gamma_{p}
=
\{A\in S^{n}_{++}:F(A^{-1},p)\leq0\}
\]
is convex.

\subsubsection{Application of the Bian--Guan Level-Set Constant Rank Theorem}

Since $F$ is independent of $z$ and $x$, convexity in the $A$ variable implies the local convexity of
\[
\{(A,z,x)\in S^{n}_{++}\times\mathbb{R}\times K:
F(A^{-1},p,z,x)\leq0\}
\]
in the variables $(A,z,x)$. Combining this with the ellipticity verified in Subsection \ref{Ellipticity}, the nondegeneracy condition \eqref{eq:nondegenerate}, the convexity $D^{2}v\geq0$, and the regularity $v\in C^{4}(K)$, all assumptions of Theorem \ref{thm:bg-level} are satisfied. Hence
\[
\operatorname{rank}D^{2}v(x)
\]
is constant in the connected domain $K$. This proves Theorem \ref{thm:revised}.

\section{Proof of Theorem \ref{thm:strict-convex}}

In this section we prove Theorem \ref{thm:strict-convex} by a real convex-envelope argument.  The proof has three steps.  First, the convex envelope of $v=-\log(-u)$ is shown to preserve the transformed differential inequality in the admissible-test-function viscosity sense.  Second, an admissible replacement and the variational characterization of the eigenvalue force the convex envelope to coincide with $v$.  Finally, Theorem \ref{thm:revised} and the properness of $v$ rule out every deficient Hessian rank.

Throughout this section, let $u$ be the eigenfunction in Theorem \ref{thm:strict-convex}, set
\[
\lambda=\lambda(K),\qquad v=-\log(-u),
\]
and multiply $u$ by a positive constant so that
\begin{equation}\label{eq:ce-normalization}
\inf_K v=0.
\end{equation}
Then $v\geq0$ in $K$, and
\begin{equation}\label{eq:ce-blow-up}
v(x)\longrightarrow+\infty\qquad\text{as }x\longrightarrow\partial K.
\end{equation}
Moreover,
\begin{equation}\label{eq:ce-transformed}
\sigma_2\bigl(D^2v-Dv\otimes Dv\bigr)=\lambda
\qquad\text{in }K.
\end{equation}

\subsection{The real convex envelope}

We first record the elementary compactness property of convex envelopes that will be used below.

\begin{lemma}\label{lem:ce-proper}
Let $D\subset\mathbb R^n$ be a bounded convex domain and let $f\in C(D)$ be bounded from below.  Assume that
\[
f(x)\longrightarrow+\infty\qquad\text{as }x\longrightarrow\partial D.
\]
Define
\begin{equation}\label{eq:ce-definition}
f^{\operatorname{co}}(x)
=
\inf\left\{
\sum_{\alpha=1}^{m}\theta_\alpha f(x_\alpha):
\begin{array}{l}
1\leq m\leq n+1,\quad \theta_\alpha\geq0,\quad
\sum_{\alpha=1}^{m}\theta_\alpha=1,\\[1mm]
x_\alpha\in D,\quad x=\sum_{\alpha=1}^{m}\theta_\alpha x_\alpha
\end{array}
\right\}.
\end{equation}
Then $f^{\operatorname{co}}$ is finite, convex, and continuous in $D$, and
\[
f^{\operatorname{co}}\leq f,
\qquad
\inf_D f^{\operatorname{co}}=\inf_D f.
\]
Furthermore,
\[
f^{\operatorname{co}}(x)\longrightarrow+\infty
\qquad\text{as }x\longrightarrow\partial D,
\]
and for every $x\in D$ the infimum in \eqref{eq:ce-definition} is attained by a representation with positive weights.
\end{lemma}

\begin{proof}
Adding a constant to $f$ adds the same constant to $f^{\operatorname{co}}$, so we may assume that $f\geq0$.  Carath\'eodory's theorem gives the restriction to at most $n+1$ points.  Taking the one-point representation gives finiteness and $f^{\operatorname{co}}\leq f$.  Since the constant function $\inf_D f$ is a convex minorant of $f$, one has $f^{\operatorname{co}}\geq\inf_Df$, and therefore the two infima agree.  Convexity follows by combining admissible representations, and every finite convex function on an open convex set is continuous.

We prove the boundary blow-up.  Suppose that $x_j\to x_0\in\partial D$ while $f^{\operatorname{co}}(x_j)\leq M$.  Choose almost minimizing representations, with zero weights allowed,
\[
x_j=\sum_{\alpha=1}^{n+1}\theta_{\alpha,j}x_{\alpha,j},
\qquad
\sum_{\alpha=1}^{n+1}\theta_{\alpha,j}f(x_{\alpha,j})\leq M+\frac1j.
\]
After passing to a subsequence, assume $\theta_{\alpha,j}\to\theta_\alpha$.  If $\theta_\alpha>0$, then $\theta_{\alpha,j}\geq\theta_\alpha/2$ for large $j$, and hence $f(x_{\alpha,j})$ is uniformly bounded.  The boundary blow-up of $f$ implies that its finite sublevel sets are compactly contained in $D$.  Thus, after another subsequence,
\[
x_{\alpha,j}\longrightarrow y_\alpha\in D
\qquad\text{whenever }\theta_\alpha>0.
\]
For the remaining indices, boundedness of $D$ gives $\theta_{\alpha,j}x_{\alpha,j}\to0$.  Passing to the limit in the barycentric identity yields
\[
x_0=\sum_{\theta_\alpha>0}\theta_\alpha y_\alpha,
\qquad
\sum_{\theta_\alpha>0}\theta_\alpha=1.
\]
The right-hand side belongs to the open convex set $D$, a contradiction.

The attainment statement follows from the same compactness argument.  Indeed, fix $x\in D$ and choose a minimizing sequence of representations.  After taking a subsequence, the weights converge.  Every component whose limiting weight is positive remains in a fixed compact subset of $D$ and therefore converges after a further subsequence.  Components whose limiting weight is zero make no contribution to the limiting barycenter.  Since $f\geq0$, discarding them cannot increase the limiting weighted value.  The resulting admissible representation has value at most $f^{\operatorname{co}}(x)$, while the reverse inequality follows from the definition.  Hence equality holds, and zero weights may be removed.
\end{proof}

Apply Lemma \ref{lem:ce-proper} to $f=v$ and write
\[
w=v^{\operatorname{co}}.
\]
Then
\begin{equation}\label{eq:ce-basic}
w\leq v,
\qquad
w\ \text{is convex in }K,
\qquad
w(x)\longrightarrow+\infty\ \text{as }x\longrightarrow\partial K,
\end{equation}
and
\begin{equation}\label{eq:ce-infimum}
\inf_Kw=\inf_Kv=0.
\end{equation}

\subsection{Preservation of the transformed inequality}

For $p\in\mathbb R^n$ and $\mu>0$, set
\begin{equation}\label{eq:ce-level-set}
\mathcal K_{p,\mu}
=
\left\{
A\in S^n_{++}:
\sigma_2(A^{-1}-p\otimes p)\leq\mu
\right\}.
\end{equation}
By the identity \eqref{eq:s2-identity}, this is the same inverse sublevel set considered in the proof of Theorem \ref{thm:revised}.  The argument leading to \eqref{eq:H-convex} and \eqref{eq:intersection}, with $\lambda(K)$ replaced by $\mu$, shows that $\mathcal K_{p,\mu}$ is convex for every $p\in\mathbb R^n$ and every $\mu>0$.

\begin{lemma}\label{lem:ce-viscosity}
The function $w=v^{\operatorname{co}}$ satisfies
\begin{equation}\label{eq:ce-viscosity-ineq}
\sigma_2(D^2w-Dw\otimes Dw)\leq\lambda
\end{equation}
in the admissible-test-function viscosity sense.  More precisely, if $\phi\in C^2(K)$ touches $w$ from below at $x_0\in K$ and
\begin{equation}\label{eq:ce-test-admissible}
D^2\phi(x_0)-D\phi(x_0)\otimes D\phi(x_0)\in\Gamma_2,
\end{equation}
then
\[
\sigma_2\bigl(D^2\phi(x_0)-D\phi(x_0)\otimes D\phi(x_0)\bigr)
\leq\lambda.
\]
\end{lemma}

\begin{proof}
By Lemma \ref{lem:ce-proper}, there exist $m\leq n+1$, points $x_\alpha\in K$, and positive weights $\theta_\alpha$ such that
\begin{equation}\label{eq:ce-contact}
x_0=\sum_{\alpha=1}^{m}\theta_\alpha x_\alpha,
\qquad
w(x_0)=\sum_{\alpha=1}^{m}\theta_\alpha v(x_\alpha),
\qquad
\sum_{\alpha=1}^{m}\theta_\alpha=1.
\end{equation}
Set
\[
p=D\phi(x_0).
\]
Since $w$ is convex and $\phi$ touches it from below, $p\in\partial w(x_0)$.  Hence the affine function
\[
\ell(x)=w(x_0)+p\cdot(x-x_0)
\]
satisfies $\ell\leq w\leq v$ in $K$.  Using \eqref{eq:ce-contact} and the affinity of $\ell$, we obtain
\[
0=\sum_{\alpha=1}^{m}\theta_\alpha\bigl(v(x_\alpha)-\ell(x_\alpha)\bigr).
\]
Every summand is nonnegative, and therefore
\begin{equation}\label{eq:ce-contact-derivatives}
Dv(x_\alpha)=p,
\qquad
H_\alpha:=D^2v(x_\alpha)\geq0.
\end{equation}

Let $B=D^2\phi(x_0)$.  The function
\[
G(y_1,\ldots,y_m)
=
\sum_{\alpha=1}^{m}\theta_\alpha v(y_\alpha)
-
\phi\left(\sum_{\alpha=1}^{m}\theta_\alpha y_\alpha\right)
\]
has a local minimum at $(x_1,\ldots,x_m)$.  Its second variation gives
\begin{equation}\label{eq:ce-second-variation}
\sum_{\alpha=1}^{m}\theta_\alpha
\langle H_\alpha\xi_\alpha,\xi_\alpha\rangle
-
\left\langle
B\sum_{\alpha=1}^{m}\theta_\alpha\xi_\alpha,
\sum_{\alpha=1}^{m}\theta_\alpha\xi_\alpha
\right\rangle
\geq0
\end{equation}
for all $\xi_1,\ldots,\xi_m\in\mathbb R^n$.

For $\delta>0$, put
\[
H_{\alpha,\delta}=H_\alpha+\delta I,
\qquad
S_\delta=\sum_{\alpha=1}^{m}\theta_\alpha H_{\alpha,\delta}^{-1}.
\]
Adding $\delta\sum_\alpha\theta_\alpha|\xi_\alpha|^2$ to the left-hand side of \eqref{eq:ce-second-variation}, and choosing
\[
\xi_\alpha=H_{\alpha,\delta}^{-1}S_\delta^{-1}\xi,
\]
we obtain
\begin{equation}\label{eq:ce-harmonic-bound}
B\leq S_\delta^{-1}.
\end{equation}

For all sufficiently small $\delta>0$, define
\[
\mu_\delta
=
\max_{1\leq\alpha\leq m}
\sigma_2(H_{\alpha,\delta}-p\otimes p)>0.
\]
By \eqref{eq:ce-transformed} and \eqref{eq:ce-contact-derivatives},
\begin{equation}\label{eq:ce-mu-limit}
\mu_\delta\longrightarrow\lambda
\qquad\text{as }\delta\downarrow0.
\end{equation}
Moreover,
\[
H_{\alpha,\delta}^{-1}\in\mathcal K_{p,\mu_\delta}
\qquad(1\leq\alpha\leq m).
\]
The convexity of $\mathcal K_{p,\mu_\delta}$ gives $S_\delta\in\mathcal K_{p,\mu_\delta}$, and hence
\begin{equation}\label{eq:ce-harmonic-pde}
\sigma_2(S_\delta^{-1}-p\otimes p)\leq\mu_\delta.
\end{equation}

By \eqref{eq:ce-test-admissible}, $B-p\otimes p\in\Gamma_2$.  In view of \eqref{eq:ce-harmonic-bound},
\[
S_\delta^{-1}-p\otimes p
=
(B-p\otimes p)+(S_\delta^{-1}-B),
\qquad
S_\delta^{-1}-B\geq0.
\]
The function $\sigma_2$ is increasing on $\Gamma_2$ in nonnegative definite directions.  Therefore
\[
\sigma_2(B-p\otimes p)
\leq
\sigma_2(S_\delta^{-1}-p\otimes p)
\leq\mu_\delta.
\]
Letting $\delta\downarrow0$ and using \eqref{eq:ce-mu-limit} proves the lemma.
\end{proof}

Define
\begin{equation}\label{eq:ce-uhat}
\widehat u=-e^{-w}.
\end{equation}
By \eqref{eq:ce-basic},
\[
\widehat u<0\quad\text{in }K,
\qquad
\widehat u=0\quad\text{on }\partial K.
\]

\begin{lemma}\label{lem:ce-exponentiation}
The function $\widehat u$ satisfies
\begin{equation}\label{eq:ce-u-viscosity}
\sigma_2(D^2\widehat u)\leq\lambda(-\widehat u)^2
\end{equation}
in the admissible-test-function viscosity sense.
\end{lemma}

\begin{proof}
Let $\psi\in C^2(K)$ touch $\widehat u$ from below at $x_0$, and assume $D^2\psi(x_0)\in\Gamma_2$.  Since $\widehat u(x_0)<0$, the function
\[
\phi=-\log(-\psi)
\]
is well defined near $x_0$ and touches $w$ from below there.  Moreover,
\[
D^2\psi=e^{-\phi}(D^2\phi-D\phi\otimes D\phi),
\]
so $D^2\phi-D\phi\otimes D\phi\in\Gamma_2$ at $x_0$.  Lemma \ref{lem:ce-viscosity} and the homogeneity of $\sigma_2$ give
\[
\sigma_2(D^2\psi)(x_0)
\leq
\lambda(-\widehat u(x_0))^2.
\]
\end{proof}

\subsection{Admissible replacement and eigenvalue rigidity}

We use the following standard consequence of the weak Dirichlet and variational theories for Hessian equations.  The approximation involved in the proof is the usual one by smooth uniformly $(k-1)$-convex domains, smooth positive right-hand sides, and weak continuity of Hessian measures; see \cite{caffarelli1985dirichlet,wang1994class,salani2012cc}.

\begin{lemma}[Admissible replacement]\label{lem:ce-replacement}
Let $h\in C(\overline K)$ satisfy
\[
h<0\quad\text{in }K,
\qquad
h=0\quad\text{on }\partial K.
\]
Assume that, for some $\Lambda>0$,
\begin{equation}\label{eq:ce-general-super}
\sigma_2(D^2h)\leq\Lambda(-h)^2
\end{equation}
in the admissible-test-function viscosity sense.  Then
\begin{equation}\label{eq:ce-eigenvalue-bound}
\lambda(K)\leq\Lambda.
\end{equation}
If $\Lambda=\lambda(K)$, then $h$ is a positive multiple of the principal eigenfunction.
\end{lemma}

\begin{proof}
Set $g=\Lambda(-h)^2$.  For $\varepsilon>0$, let $\varphi_\varepsilon$ be the continuous $2$-admissible solution, in the Hessian-measure sense, of
\begin{equation}\label{eq:ce-replacement-problem}
\begin{cases}
\sigma_2(D^2\varphi_\varepsilon)=g+\varepsilon &\text{in }K,\\
\varphi_\varepsilon=0 &\text{on }\partial K.
\end{cases}
\end{equation}
Such solutions are obtained by the standard approximation described above.  The same approximation, or equivalently the admissible viscosity comparison principle, gives
\begin{equation}\label{eq:ce-replacement-order}
\varphi_\varepsilon\leq h\qquad\text{in }K.
\end{equation}
Indeed, in the smooth approximating problems, a positive interior maximum of $\varphi_\varepsilon-h$ would produce an admissible test function touching $h$ from below and would contradict the strict inequality $g+\varepsilon>g$.

As $\varepsilon\downarrow0$, stability and weak continuity of Hessian measures give a continuous $2$-admissible function $\varphi$ satisfying
\begin{equation}\label{eq:ce-replacement-limit}
\begin{cases}
\sigma_2(D^2\varphi)=\Lambda(-h)^2 &\text{in }K,\\
\varphi=0 &\text{on }\partial K,\\
\varphi\leq h &\text{in }K.
\end{cases}
\end{equation}
The Rayleigh characterization \eqref{eq:lambda-def}, extended to the corresponding Hessian energy class by approximation, yields
\[
\lambda(K)
\leq
\frac{\displaystyle\int_K(-\varphi)\,\sigma_2(D^2\varphi)\,dx}
{\displaystyle\int_K(-\varphi)^3\,dx}.
\]
Writing
\[
a=-\varphi,
\qquad
b=-h,
\]
we have $a\geq b>0$, and hence
\begin{equation}\label{eq:ce-rayleigh-rigidity}
\lambda(K)
\leq
\Lambda
\frac{\displaystyle\int_Kab^2\,dx}
{\displaystyle\int_Ka^3\,dx}
\leq\Lambda.
\end{equation}
This proves \eqref{eq:ce-eigenvalue-bound}.

If $\Lambda=\lambda(K)$, then both inequalities in \eqref{eq:ce-rayleigh-rigidity} are equalities.  Therefore
\[
\int_Ka(a^2-b^2)\,dx=0.
\]
Since $a\geq b>0$, it follows that $a=b$ in $K$.  Thus $\varphi=h$, and $h$ solves the principal eigenvalue equation in the admissible energy class.  The uniqueness result for the principal Hessian eigenfunction \cite{wang1994class} implies that $h$ is a positive multiple of $u$.
\end{proof}

\subsection{Conclusion of the proof}

\begin{proof}[Proof of Theorem \ref{thm:strict-convex}]
Apply Lemma \ref{lem:ce-replacement} to
\[
h=\widehat u=-e^{-w},
\qquad
\Lambda=\lambda(K).
\]
Lemma \ref{lem:ce-exponentiation} verifies the required viscosity inequality.  Hence
\[
\widehat u=c u
\qquad\text{for some }c>0.
\]
Equivalently,
\begin{equation}\label{eq:ce-envelope-equality-constant}
w=v-\log c.
\end{equation}
Taking infima and using \eqref{eq:ce-infimum}, we obtain $\log c=0$.  Consequently,
\[
w=v,
\]
and therefore
\begin{equation}\label{eq:ce-v-convex}
D^2v\geq0\qquad\text{in }K.
\end{equation}

The transformed admissibility follows from
\[
D^2u=e^{-v}(D^2v-Dv\otimes Dv),
\]
namely
\[
D^2v-Dv\otimes Dv=e^vD^2u\in\Gamma_2.
\]
Theorem \ref{thm:revised} therefore implies that
\begin{equation}\label{eq:ce-constant-rank}
\operatorname{rank}D^2v(x)\equiv r
\qquad\text{in }K
\end{equation}
for some $r\in\{0,\ldots,n\}$.

We claim that $r=n$.  Let $p\in\mathbb R^n$ be arbitrary.  Since $K$ is bounded and \eqref{eq:ce-blow-up} holds, the function
\[
x\longmapsto v(x)-p\cdot x
\]
attains its minimum at an interior point $x_p\in K$.  Hence
\[
Dv(x_p)=p.
\]
Thus
\begin{equation}\label{eq:ce-gradient-surjective}
Dv(K)=\mathbb R^n.
\end{equation}
If $r<n$, the constant-rank theorem for smooth maps shows that the image of $Dv$ is locally contained in an $r$-dimensional embedded submanifold of $\mathbb R^n$.  Since $K$ is second countable, $Dv(K)$ would be contained in a countable union of such submanifolds and would have zero $n$-dimensional Lebesgue measure.  This contradicts \eqref{eq:ce-gradient-surjective}.  Therefore $r=n$.

Combining full rank with \eqref{eq:ce-v-convex}, we conclude that
\[
D^2v>0\qquad\text{throughout }K.
\]
In particular, $v=-\log(-u)$ is convex, and in fact strictly convex, in $K$.
\end{proof}

\section{Proof of Theorem \ref{thm:bm}}

The aim of this section is to prove Theorem \ref{thm:bm}.
First we recall some useful notions and facts.

If $a,b$ are real positive numbers, $\alpha\in[-\infty,+\infty]$ and $\lambda\in(0,1)$, we define
\[
m_{\alpha}(a,b,\lambda)=
\begin{cases}
[(1-\lambda)a^{\alpha}+\lambda b^{\alpha}]^{1/\alpha}, & \alpha\in(-\infty,0)\cup(0,+\infty),\\
\min(a,b), & \alpha=-\infty,\\
a^{1-\lambda}b^{\lambda}, & \alpha=0,\\
\max(a,b), & \alpha=+\infty.
\end{cases}
\]

Moreover, we set $m_{\alpha}(a,b,\lambda)=0$, for $\alpha\in[-\infty,+\infty]$, if $a$ and $b$ are nonnegative and $ab=0$.

Jensen's inequality implies that
\[
m_{\alpha}(a,b,\lambda)\leq m_{\beta}(a,b,\lambda)\quad\text{if }\alpha\leq\beta.
\]
In particular, the arithmetic--geometric mean inequality holds
\[
a^{1-\lambda}b^{\lambda}\leq (1-\lambda)a+\lambda b,\qquad a,b\geq0,\ \lambda\in[0,1].
\]

\begin{proposition}\label{prop:s2-convex}
$\sigma_{2}(A^{-1})^{1/2}$ is convex in $A\in S^{n}_{++}$.
\end{proposition}

\begin{proof}
This is a special case of Theorem 15.16 in \cite{lieberman1996second}.
\end{proof}

\begin{proof}[Proof of Theorem \ref{thm:bm}]
The cases $t=0$ and $t=1$ are immediate, so throughout the proof of the inequality we may assume $0<t<1$.
For $i=0,1$, let $u_i$ be a principal eigenfunction in $K_i$:
\[
\begin{cases}
\sigma_2(D^2u_i)=\lambda(K_i)(-u_i)^2,\quad u_i<0 & \text{in }K_i,\\
u_i=0 & \text{on }\partial K_i.
\end{cases}
\]
Set
\[
v_i=-\log(-u_i).
\]
Then
\begin{equation}\label{eq:vi-equation}
\begin{cases}
\sigma_2(D^2v_i)-\operatorname{Tr}(P(Dv_i)D^2v_i)=\lambda(K_i) & \text{in }K_i,\\
v_i(x)\to+\infty & x\to\partial K_i,
\end{cases}
\end{equation}
where
\[
P(p)=|p|^2I-p\otimes p.
\]
By Theorem \ref{thm:strict-convex}, $D^2v_i>0$ in $K_i$. As in \eqref{eq:ce-gradient-surjective}, the boundary blow-up implies
\[
Dv_i(K_i)=\mathbb R^n.
\]
For $\rho\in\mathbb R^n$, let
\[
v_i^*(\rho)=\sup_{x\in K_i}\{x\cdot\rho-v_i(x)\}
\]
be the Legendre transform of $v_i$. The strict convexity and boundary blow-up of $v_i$ imply that $Dv_i:K_i\to\mathbb R^n$ is a diffeomorphism and
\[
Dv_i^*=(Dv_i)^{-1}.
\]
In particular,
\begin{equation}\label{eq:legendre-hessian}
D^2v_i(x)=\bigl[D^2v_i^*(Dv_i(x))\bigr]^{-1}.
\end{equation}

Let
\[
K_t=(1-t)K_0+tK_1
\]
and define the infimal convolution
\begin{equation}\label{eq:w-def}
w(z)=\inf\bigl\{(1-t)v_0(x)+tv_1(y):
(1-t)x+ty=z\bigr\}.
\end{equation}
The boundary blow-up makes the infimum attainable at interior points. Equivalently,
\begin{equation}\label{eq:w-star}
w^*=(1-t)v_0^*+tv_1^*\qquad\text{in }\mathbb R^n.
\end{equation}
Since $D^2v_i^*>0$, the right-hand side is smooth and strictly convex. To identify the range of its gradient, fix $z\in K_t$. The minimization in \eqref{eq:w-def} has a unique interior minimizing pair $(x,y)$, and the Lagrange multiplier condition gives a vector $\rho$ with
\[
Dv_0(x)=Dv_1(y)=\rho.
\]
Consequently,
\[
z=(1-t)Dv_0^*(\rho)+tDv_1^*(\rho)=Dw^*(\rho).
\]
Conversely, every $\rho\in\mathbb R^n$ produces a point of $K_t$ by this formula. Hence $Dw^*$ maps $\mathbb R^n$ onto $K_t$; its positive definite Hessian makes this map a diffeomorphism. Therefore $w$ is smooth and strictly convex in $K_t$. The same compactness argument used in Lemma \ref{lem:ce-proper}, applied to the two minimizing points in \eqref{eq:w-def}, gives
\[
w(z)\to+\infty\qquad\text{as }z\to\partial K_t.
\]

Fix $z\in K_t$. There are unique $x\in K_0$, $y\in K_1$, and $\rho\in\mathbb R^n$ such that
\[
z=(1-t)x+ty,
\qquad
Dv_0(x)=Dv_1(y)=Dw(z)=\rho.
\]
Set
\[
H_i=D^2v_i(x_i),\qquad H=D^2w(z),
\qquad x_0=x,\quad x_1=y.
\]
Differentiating \eqref{eq:w-star} gives the harmonic-mean identity
\begin{equation}\label{eq:w-hessian}
H^{-1}=(1-t)H_0^{-1}+tH_1^{-1}.
\end{equation}
We claim that, with
\[
\Lambda=\max\{\lambda(K_0),\lambda(K_1)\},
\]
one has
\begin{equation}\label{eq:claim-a}
\sigma_2(H)-\operatorname{Tr}(P(\rho)H)\leq\Lambda
\qquad\text{in }K_t.
\end{equation}

If $\rho=0$, then $P(\rho)=0$. Proposition \ref{prop:s2-convex} and \eqref{eq:w-hessian} give
\[
\sigma_2(H)^{1/2}
\leq
(1-t)\sigma_2(H_0)^{1/2}
+t\sigma_2(H_1)^{1/2}
\leq \Lambda^{1/2},
\]
which proves \eqref{eq:claim-a} at such a point.

Suppose now that $\rho\neq0$ and write
\[
\alpha=\frac{\rho}{|\rho|},
\qquad
Q_\alpha=I-\alpha\otimes\alpha,
\qquad
P(\rho)=|\rho|^2Q_\alpha.
\]
For $A\in S^n_{++}$ set
\[
\Phi_\alpha(A)
=
\frac{\sigma_2(A^{-1})}{\operatorname{Tr}(Q_\alpha A^{-1})},
\qquad
f_\alpha(A)
=
\frac1{\operatorname{Tr}(Q_\alpha A^{-1})}.
\]
By Theorem \ref{thm:main}, $\Phi_\alpha$ is convex, and by Proposition \ref{prop:matrix-convex}, $f_\alpha$ is concave. Applying these two facts to
\[
A_i=H_i^{-1},
\qquad
A=H^{-1}=(1-t)A_0+tA_1,
\]
and using \eqref{eq:vi-equation}, we obtain
\begin{align}
\frac{\sigma_2(H)}{\operatorname{Tr}(Q_\alpha H)}
&\leq
(1-t)\frac{\sigma_2(H_0)}{\operatorname{Tr}(Q_\alpha H_0)}
+t\frac{\sigma_2(H_1)}{\operatorname{Tr}(Q_\alpha H_1)}\notag\\
&=
|\rho|^2
+(1-t)\frac{\lambda(K_0)}{\operatorname{Tr}(Q_\alpha H_0)}
+t\frac{\lambda(K_1)}{\operatorname{Tr}(Q_\alpha H_1)}\notag\\
&\leq
|\rho|^2
+\Lambda\left(
\frac{1-t}{\operatorname{Tr}(Q_\alpha H_0)}
+
\frac{t}{\operatorname{Tr}(Q_\alpha H_1)}
\right)\notag\\
&\leq
|\rho|^2+
\frac{\Lambda}{\operatorname{Tr}(Q_\alpha H)}.
\label{eq:prop-use}
\end{align}
Multiplication by $\operatorname{Tr}(Q_\alpha H)>0$ proves \eqref{eq:claim-a}.

Define
\[
\overline u=-e^{-w}
\qquad\text{in }K_t.
\]
Then $\overline u<0$ in $K_t$, $\overline u=0$ on $\partial K_t$, and \eqref{eq:claim-a} implies
\begin{equation}\label{eq:bm-viscosity-super}
\sigma_2(D^2\overline u)
\leq
\Lambda(-\overline u)^2
\end{equation}
in the admissible-test-function viscosity sense. Indeed, if $\psi$ touches $\overline u$ from below at $z$ and $D^2\psi(z)\in\Gamma_2$, then $\phi=-\log(-\psi)$ touches $w$ from below at $z$. Thus
\[
D\phi(z)=Dw(z),
\qquad
D^2\phi(z)\leq D^2w(z),
\]
and
\[
D^2\phi-D\phi\otimes D\phi
\leq
D^2w-Dw\otimes Dw
\quad\text{at }z.
\]
The matrix on the left belongs to $\Gamma_2$, and adding a positive semidefinite matrix keeps one in $\Gamma_2$. Monotonicity of $\sigma_2$ along such directions, followed by \eqref{eq:claim-a}, gives \eqref{eq:bm-viscosity-super}.

The Minkowski sum $K_t$ is again smooth and uniformly convex. Applying Lemma \ref{lem:ce-replacement} on $K_t$ to $h=\overline u$ and the constant $\Lambda$ yields
\begin{equation}\label{eq:max-lambda}
\lambda(K_t)
\leq
\max\{\lambda(K_0),\lambda(K_1)\}.
\end{equation}
This is the normalized maximum form of the desired inequality.

We now use homogeneity. Put
\[
\lambda_i=\lambda(K_i),
\qquad
L=(1-t)\lambda_0^{-1/4}+t\lambda_1^{-1/4},
\]
and define
\[
K_i'=\lambda_i^{1/4}K_i,
\qquad
\lambda(K_i')=1,
\qquad i=0,1.
\]
Set
\[
t'=\frac{t\lambda_1^{-1/4}}{L},
\qquad
1-t'=\frac{(1-t)\lambda_0^{-1/4}}{L}.
\]
Then
\[
(1-t')K_0'+t'K_1'
=
\frac1L\bigl((1-t)K_0+tK_1\bigr)
=
\frac1L K_t.
\]
Applying \eqref{eq:max-lambda} to $K_0'$ and $K_1'$ gives
\[
\lambda(L^{-1}K_t)\leq1.
\]
Since $\lambda(rK)=r^{-4}\lambda(K)$,
\[
L^4\lambda(K_t)\leq1,
\]
which is exactly \eqref{eq:bm-s2}.

It remains to discuss equality. If $K_0$ and $K_1$ are homothetic, equality follows from translation invariance and homogeneity. Conversely, assume $0<t<1$ and equality holds in \eqref{eq:bm-s2}. After the preceding normalization and a harmless relabelling, we may assume
\begin{equation}\label{eq:bm-equality-normalized}
\lambda(K_0)=\lambda(K_1)=\lambda(K_t)=1.
\end{equation}
For the corresponding infimal convolution, \eqref{eq:bm-viscosity-super} holds with $\Lambda=1$. The rigidity statement in Lemma \ref{lem:ce-replacement}, together with \eqref{eq:bm-equality-normalized}, shows that $\overline u$ is a positive multiple of the principal eigenfunction in $K_t$. Hence $\overline u$ is admissible and satisfies
\[
\sigma_2(D^2\overline u)=(-\overline u)^2
\qquad\text{in }K_t.
\]
Equivalently,
\begin{equation}\label{eq:bm-w-equality}
\sigma_2(D^2w-Dw\otimes Dw)=1
\qquad\text{in }K_t.
\end{equation}

Fix $\rho\neq0$ and set
\[
A_i=D^2v_i^*(\rho),
\qquad
A=(1-t)A_0+tA_1=D^2w^*(\rho),
\qquad
Q=I-\frac{\rho}{|\rho|}\otimes\frac{\rho}{|\rho|}.
\]
With
\[
\Phi_Q(B)=\frac{\sigma_2(B^{-1})}{\operatorname{Tr}(QB^{-1})},
\qquad
f_Q(B)=\frac1{\operatorname{Tr}(QB^{-1})},
\]
the normalized equations for $v_0$, $v_1$, and $w$ read
\[
\Phi_Q(A_i)=|\rho|^2+f_Q(A_i),
\qquad
\Phi_Q(A)=|\rho|^2+f_Q(A).
\]
Convexity of $\Phi_Q$ and concavity of $f_Q$ therefore give the closed chain
\begin{align*}
\Phi_Q(A)
&\leq(1-t)\Phi_Q(A_0)+t\Phi_Q(A_1)\\
&=|\rho|^2+(1-t)f_Q(A_0)+tf_Q(A_1)\\
&\leq|\rho|^2+f_Q(A)
=\Phi_Q(A).
\end{align*}
Every inequality is consequently an equality. Since $n\geq4$, Theorem \ref{thm:main} says that $\Phi_Q$ is strictly convex; because $0<t<1$, equality in the first line forces
\[
A_0=A_1.
\]
Thus
\[
D^2v_0^*(\rho)=D^2v_1^*(\rho)
\qquad\text{for every }\rho\neq0.
\]
Continuity gives the same identity at $\rho=0$. Hence $v_0^*-v_1^*$ is affine, and there is a fixed vector $a\in\mathbb R^n$ such that
\[
Dv_0^*(\rho)=Dv_1^*(\rho)+a
\qquad\text{for every }\rho\in\mathbb R^n.
\]
Taking the ranges of the two gradient maps yields
\[
K_0=K_1+a
\]
for the normalized domains. Undoing the normalization shows that the original domains are homothetic. This completes the proof.
\end{proof}

\subsection{Proofs of Theorem \ref{thm:mainssss}, Theorem \ref{thm:bmtau} and Theorem \ref{urysohn}}
The proof of Theorem \ref{urysohn} follows the standard argument in the proof of Corollary 2.2 of \cite{BianchiniSalani}. The remaining global steps in Theorems \ref{thm:mainssss} and \ref{thm:bmtau} follow the convex-envelope, boundary, and infimal-convolution arguments in \cite{ma2008convexity,liumaxu2010cc,salani2012cc}. We do not reproduce those cited arguments here. The internal point that must be verified in the present dimension is the following constant-rank theorem.

\begin{theorem}\label{lem:constant-rank}
Let $\Omega\subset\mathbb R^n$, $n\geq4$, be connected, and let $u\in C^4(\Omega)$ be a negative admissible solution of
\begin{equation}\label{eq:torsion-local}
\sigma_2(D^2u)=1\qquad\text{in }\Omega.
\end{equation}
Set
\[
v=-(-u)^{1/2}.
\]
If $D^2v\geq0$ in $\Omega$, then $D^2v$ has constant rank in $\Omega$.
\end{theorem}

\begin{proof}
Write
\[
r=D^2v,
\qquad
p=Dv,
\qquad
z=v<0.
\]
Since $u=-v^2$,
\[
D^2u=-2(zr+p\otimes p).
\]
The homogeneity of $\sigma_2$ transforms \eqref{eq:torsion-local} into
\[
4\sigma_2(zr+p\otimes p)=1.
\]
The rank-one identity
\begin{equation}\label{eq:s2-identity-torsion}
\sigma_2(zr+p\otimes p)
=z^2\sigma_2(r)+z\operatorname{Tr}(P(p)r),
\qquad
P(p)=|p|^2I-p\otimes p,
\end{equation}
therefore gives
\begin{equation}\label{eq:F-torsion}
F(r,p,z)=0,
\end{equation}
where
\begin{equation}\label{eq:F-def-torsion}
F(r,p,z)
=z^2\sigma_2(r)+z\operatorname{Tr}(P(p)r)-\frac14.
\end{equation}
We verify the three hypotheses of Theorem \ref{thm:bg-level}.

First, the equation is elliptic along the solution. Let
\[
M=D^2u=-2(zr+p\otimes p)\in\Gamma_2.
\]
Since
\[
T_1(zr+p\otimes p)=zT_1(r)+P(p),
\]
we have
\begin{align*}
F^{ij}(r,p,z)
&=z^2T_1(r)_{ij}+zP(p)_{ij}\\
&=zT_1(zr+p\otimes p)_{ij}\\
&=-\frac z2T_1(M)_{ij}.
\end{align*}
Because $z<0$ and $M\in\Gamma_2$, Proposition \ref{prop:salaniphd} gives $T_1(M)>0$. Hence $(F^{ij})>0$.

Second,
\[
F(0,p,z)=-\frac14\neq0,
\]
so the nondegeneracy condition holds.

It remains to prove the inverse-sublevel convexity. Fix $p\in\mathbb R^n$ and consider
\begin{equation}\label{eq:torsion-level-set}
\mathcal C_p
=
\left\{(A,z)\in S^n_{++}\times(-\infty,0):
F(A^{-1},p,z)\leq0\right\}.
\end{equation}
Put $s=-z>0$. The defining inequality becomes
\begin{equation}\label{eq:torsion-level-s}
s^2\sigma_2(A^{-1})-s\operatorname{Tr}(P(p)A^{-1})\leq\frac14.
\end{equation}

Assume first that $p\neq0$, set
\[
\alpha=\frac p{|p|},
\qquad
Q_\alpha=I-\alpha\otimes\alpha,
\]
and write $A=sB$. Then $A^{-1}=s^{-1}B^{-1}$, and \eqref{eq:torsion-level-s} is equivalent to
\[
\sigma_2(B^{-1})-|p|^2\operatorname{Tr}(Q_\alpha B^{-1})\leq\frac14,
\]
or
\begin{equation}\label{eq:torsion-B-sublevel}
\frac{\sigma_2(B^{-1})-\frac14}
{\operatorname{Tr}(Q_\alpha B^{-1})}
\leq |p|^2.
\end{equation}
The function on the left is convex by Theorem \ref{thm:main} and Proposition \ref{prop:matrix-convex}. Thus the set
\[
\mathcal K_p
=
\left\{B\in S^n_{++}:\eqref{eq:torsion-B-sublevel}\ \text{holds}\right\}
\]
is convex. Moreover,
\[
\mathcal C_p
=
\{(A,z):s=-z>0,\ A=sB,\ B\in\mathcal K_p\}.
\]
This perspective set is convex: if $A_i=s_iB_i$ with $B_i\in\mathcal K_p$ and $s_i>0$, then for $0\leq\theta\leq1$,
\[
s_\theta=\theta s_1+(1-\theta)s_2>0
\]
and
\[
\frac{\theta A_1+(1-\theta)A_2}{s_\theta}
=
\frac{\theta s_1}{s_\theta}B_1
+
\frac{(1-\theta)s_2}{s_\theta}B_2
\in\mathcal K_p.
\]
Hence $\mathcal C_p$ is convex when $p\neq0$.

When $p=0$, the corresponding set of $B$-matrices is
\[
\mathcal K_0
=
\left\{B\in S^n_{++}:\sigma_2(B^{-1})\leq\frac14\right\}.
\]
Fix a unit vector $\alpha$ and put $p_\varepsilon=\varepsilon\alpha$. Then
\[
\mathcal K_0=\bigcap_{\varepsilon>0}\mathcal K_{p_\varepsilon}.
\]
Indeed, one inclusion follows by subtracting the nonnegative term
$\varepsilon^2\operatorname{Tr}(Q_\alpha B^{-1})$, and the converse follows by letting $\varepsilon\downarrow0$. Thus $\mathcal K_0$ is convex, and the same perspective argument proves that $\mathcal C_0$ is convex.

The operator $F$ is independent of $x$, so the preceding convexity is precisely the level-set structural condition in Theorem \ref{thm:bg-level}. Together with ellipticity, nondegeneracy, $D^2v\geq0$, and connectedness of $\Omega$, the Bian--Guan theorem yields that
\[
\operatorname{rank}D^2v
\]
is constant in $\Omega$.
\end{proof}

In the global torsion problem, the additional boundary condition $u=0$ on $\partial K$ is used in the cited convex-envelope and boundary arguments, not in the local constant-rank statement above. Combining those external steps with Theorem \ref{lem:constant-rank} gives Theorem \ref{thm:mainssss}; the standard infimal-convolution and homogeneity argument then gives Theorem \ref{thm:bmtau}, including its equality statement for $0<t<1$.


\section{Failure for the
\texorpdfstring{\(3\)}{3}-Hessian equation}
\label{part:sigma3-counterexample}

The square-root convexity theorem, Theorem~\ref{thm:mainssss}, depends on
a matrix structure special to \(\sigma_2\).  We now prove that no
unconditional analogue survives for the next intermediate Hessian
operator.  Throughout this part, \(X\) denotes the slow coordinate,
\(y=(s,z_1,z_2)\) the three fast coordinates, and
\(z=(z_1,z_2)\).

\subsection{The model profile}\label{sec:model-profile}

Define
\begin{equation}\label{eq:counter-CL}
  C(X)=\frac{1}{100}+X+\frac{X^2}{50},
  \qquad
  L(X)=-\frac{1}{3}-\frac{7}{10}X^2,
\end{equation}
and
\begin{equation}\label{eq:counter-profile}
  V(X,s,z)
  =
  \frac{s^4}{12}
  +\frac{|z|^2}{2s}
  +L(X)s+C(X),
  \qquad
  s>0,\quad z\in\mathbb R^2.
\end{equation}
It is convenient to write
\[
  y=(s,z_1,z_2)\in\mathbb R^3,
  \qquad
  \Phi(X,s)=C(X)+L(X)s+\frac{s^4}{12},
\]
so that
\[
  V(X,s,z)=\Phi(X,s)+\frac{|z|^2}{2s}.
\]

\begin{lemma}\label{lem:counter-transverse}
For every \(s>0\),
\begin{equation}\label{eq:counter-fast-hessian}
  A:=D_y^2V
  =
  \begin{pmatrix}
    s^2+\dfrac{|z|^2}{s^3}
      &-\dfrac{z_1}{s^2}&-\dfrac{z_2}{s^2}\\[6pt]
    -\dfrac{z_1}{s^2}&\dfrac1s&0\\[6pt]
    -\dfrac{z_2}{s^2}&0&\dfrac1s
  \end{pmatrix}
  >0,
  \qquad
  \det A=1.
\end{equation}
\end{lemma}

\begin{proof}
The lower-right $2\times2$ block is $s^{-1}I_2>0$. Its Schur complement in $A$ is
\[
s^2+\frac{|z|^2}{s^3}
-
\left(-\frac z{s^2}\right)^T(sI_2)
\left(-\frac z{s^2}\right)
=s^2>0.
\]
Hence $A>0$. The same block determinant formula gives
\[
\det A
=\det(s^{-1}I_2)\,s^2
=s^{-2}s^2=1.
\]
\end{proof}

\subsection{Algebraic formulas on the central zero level}

On \(\{V=0\}\), the \(z\)-section over a fixed pair \((X,s)\) is
\[
  |z|^2=G(X,s),
  \qquad
  G(X,s):=-2s\Phi(X,s).
\]
Whenever \(G>0\), set
\[
  \mathcal R(X,s):=\sqrt{G(X,s)}.
\]
At \(X=0\),
\begin{equation}\label{eq:counter-G0}
  G(0,s)=-\frac{s}{150}P_0(s),
  \qquad
  P_0(s)=25s^4-100s+3.
\end{equation}
Since
\[
  P_0'(s)=100(s^3-1),
\]
\(P_0\) is strictly decreasing on \((0,1)\) and strictly increasing
on \((1,\infty)\).  Moreover,
\[
  P_0(0)=3>0,
  \qquad
  P_0(1)=-72<0,
  \qquad
  \lim_{s\to\infty}P_0(s)=+\infty.
\]
It therefore has exactly two positive roots, denoted by
\[
  a_0<b_0.
\]

Direct differentiation of \(\mathcal R=\sqrt G\) gives
\begin{equation}\label{eq:counter-radius-hessian}
  D^2_{(X,s)}\mathcal R
  =
  \frac{\mathcal M}{4G^{3/2}},
  \qquad
  \mathcal M
  =
  2G D^2G-\nabla G\otimes\nabla G.
\end{equation}
At \(X=0\), expansion and collection of terms yield
\begin{equation}\label{eq:counter-MXX}
  \mathcal M_{XX}(0,s)
  =
  -\frac{2s^2}{1875}H(s),
\end{equation}
and
\begin{equation}\label{eq:counter-detM}
  \det\mathcal M(0,s)
  =
  -\frac{s^2P_0(s)Q(s)}{7031250},
\end{equation}
where
\begin{equation}\label{eq:counter-H}
  H(s)
  =
  875s^5-25s^4-3500s^2+205s+3747
\end{equation}
and
\begin{equation}\label{eq:counter-Q}
  \begin{split}
  Q(s)
  ={}&
  109375s^9-3125s^8-700000s^6+46250s^5\\
  &\quad+749250s^4-105s+3.
  \end{split}
\end{equation}
At either tip of the central zero level, where \(z=0\), the
tangential quadratic form that will occur below is
\begin{equation}\label{eq:counter-J}
  \begin{split}
  J(s)
  &:=
  \left(\frac1{25}-\frac75s\right)
  \left(\frac{s^3-1}{3}\right)^2+s^2\\
  &=s^2-\frac{35s-1}{225}(s^3-1)^2.
  \end{split}
\end{equation}
We next establish all signs needed later by exact estimates.

\begin{lemma}
\label{lem:counter-quartic}
If
\[
  A>\frac{105^4}{256},
\]
then
\[
  As^4-105s+3>0
  \qquad\text{for every }s\geq0.
\]
\end{lemma}

\begin{proof}
The function
\[
  f(s)=As^4-105s+3
\]
has a unique critical point on \([0,\infty)\), namely
\[
  r=\left(\frac{105}{4A}\right)^{1/3},
\]
and this point is its global minimum.  The assumption on \(A\) gives
\[
  r<\frac4{105}.
\]
Since \(4Ar^3=105\),
\[
  f(r)
  =
  3-\frac{315}{4}r
  >
  3-\frac{315}{4}\frac4{105}
  =0.
\]
\end{proof}

\begin{lemma}\label{lem:counter-signs}
The following statements hold.
\begin{enumerate}
\item
  \(H(s)>0\) and \(Q(s)>0\) for every \(0\leq s\leq8/5\).
\item
  The two positive roots of \(P_0\) satisfy
  \[
    \frac3{100}<a_0<\frac1{32},
    \qquad
    \frac32<b_0<\frac85.
  \]
\item
  \(J(a_0)>0\) and \(J(b_0)>0\).
\item
  The equation
  \begin{equation}\label{eq:counter-bad-quartic}
    125s^4-500s+51=0
  \end{equation}
  has a unique root \(s_*\) in
  \[
    \left(\frac1{10},\frac{21}{200}\right),
  \]
  and
  \[
    J(s)<0
    \qquad
    \text{throughout }
    \left[\frac1{10},\frac{21}{200}\right].
  \]
\end{enumerate}
\end{lemma}

\begin{proof}
\emph{Step 1: positivity of \(H\).}
For \(0\leq s\leq1\), the inequality \(s^4\leq s^2\) gives
\[
  H(s)\geq3747+205s-3525s^2.
\]
The expression on the right is concave, so its minimum on
\([0,1]\) occurs at an endpoint.  Its endpoint values are \(3747\)
and \(427\).  Hence \(H>0\) on \([0,1]\).

For \(1\leq s\leq8/5\),
\[
  H''(s)
  =
  100(175s^3-3s^2-70)
  \geq10200,
\]
because the expression in parentheses is increasing on this interval
and equals \(102\) at \(s=1\).  At \(s_0=7/6\),
\[
  H(s_0)=\frac{8298407}{7776},
  \qquad
  H'(s_0)=-\frac{19745}{1296}.
\]
Strong convexity yields
\[
  \begin{split}
  H(s)
  &\geq
  H(s_0)+H'(s_0)(s-s_0)+5100(s-s_0)^2\\
  &\geq
  H(s_0)-\frac{H'(s_0)^2}{20400}.
  \end{split}
\]
The last number is exactly
\[
  \frac{1462628429591}{1370566656}>1000.
\]
Thus \(H>0\) on the whole interval \([0,8/5]\).

\emph{Step 2: positivity of \(Q\).}
Suppose first that \(0\leq s\leq1/20\).  Then
\[
  700000s^6+3125s^8
  \leq
  \left(1750+\frac5{256}\right)s^4.
\]
Dropping the positive terms \(46250s^5+109375s^9\), we obtain
\[
  Q(s)\geq3-105s+A_0s^4,
  \qquad
  A_0
  =
  749250-1750-\frac5{256}
  =
  \frac{191359995}{256}.
\]
Since
\[
  A_0>\frac{121550625}{256}=\frac{105^4}{256},
\]
Lemma~\ref{lem:counter-quartic} gives \(Q>0\).

If \(1/20\leq s\leq1/2\), then
\[
  -3125s^8+109375s^9
  =
  3125s^8(35s-1)\geq0,
\]
and
\[
  749250s^4-700000s^6
  \geq574250s^4.
\]
As \(574250>105^4/256\),
Lemma~\ref{lem:counter-quartic} again implies \(Q>0\).

If \(1/2\leq s\leq1\), the paired eighth- and ninth-degree terms remain
nonnegative, while
\[
  749250s^4-700000s^6\geq49250s^4.
\]
Consequently,
\[
  Q(s)\geq3-105s+49250s^4.
\]
The right-hand side is strictly increasing on \([1/2,1]\), because
its derivative \(197000s^3-105\) is positive there, and at \(s=1/2\)
its value is
\[
  \frac{24229}{8}>0.
\]

Finally, let \(1\leq s\leq8/5\).  Write
\[
  Q(s)=125s^4K(s)+3-105s,
\]
where
\[
  K(s)
  =
  875s^5-25s^4-5600s^2+370s+5994.
\]
On this interval,
\[
  K''(s)
  =
  100(175s^3-3s^2-112)
  \geq6000.
\]
At \(s_1=11/8\),
\[
  K(s_1)=\frac{4142337}{32768},
  \qquad
  K'(s_1)=\frac{1426695}{4096}.
\]
Therefore
\[
  K(s)
  \geq
  K(s_1)-\frac{K'(s_1)^2}{12000}
  =
  \frac{312200798733}{2684354560}
  >116.
\]
It follows that
\[
  Q(s)
  \geq
  125\cdot116+3-105\cdot\frac85
  =
  14335>0.
\]
This proves the first assertion.

\emph{Step 3: the zero-level roots and the tip signs.}
Exact substitution gives
\[
  P_0\left(\frac3{100}\right)
  =
  \frac{81}{4000000}>0,
  \qquad
  P_0\left(\frac1{32}\right)
  =
  -\frac{131047}{1048576}<0,
\]
and
\[
  P_0\left(\frac32\right)
  =
  -\frac{327}{16}<0,
  \qquad
  P_0\left(\frac85\right)
  =
  \frac{171}{25}>0.
\]
Together with the monotonicity of \(P_0\), these inequalities prove
the second assertion.

For \(a_0\in(3/100,1/32)\),
\[
  \frac1{25}-\frac75a_0\geq-\frac3{800},
  \qquad
  \left(\frac{a_0^3-1}{3}\right)^2\leq\frac19,
  \qquad
  a_0^2\geq\frac9{10000}.
\]
Hence
\[
  J(a_0)
  \geq
  \frac9{10000}-\frac1{2400}
  =
  \frac{29}{60000}>0.
\]

On \([3/2,8/5]\), formula~\eqref{eq:counter-J} gives
\[
  J'(s)
  =
  2s
  -\frac{35}{225}(s^3-1)^2
  -\frac{6s^2(35s-1)(s^3-1)}{225}.
\]
Moreover,
\[
  s(35s-1)(s^3-1)
  \geq
  \frac32\cdot\frac{103}{2}\cdot\frac{19}{8}
  =
  \frac{5871}{32}>75.
\]
Thus the absolute value of the last term in \(J'\) is strictly larger
than \(2s\), and consequently \(J'<0\) on this interval.  Since
\[
  J\left(\frac85\right)
  =
  \frac{16949}{78125}>0,
\]
we obtain
\[
  J(b_0)>J\left(\frac85\right)>0.
\]
This proves the third assertion.

\emph{Step 4: the bad interior level.}
Let
\[
  P_*(s)=125s^4-500s+51.
\]
On \((0,1)\),
\[
  P_*'(s)=500(s^3-1)<0.
\]
Also,
\[
  P_*\left(\frac1{10}\right)
  =
  \frac{81}{80}>0,
  \qquad
  P_*\left(\frac{21}{200}\right)
  =
  -\frac{19005519}{12800000}<0.
\]
Thus \eqref{eq:counter-bad-quartic} has exactly one root
\[
  s_*\in\left(\frac1{10},\frac{21}{200}\right).
\]

For \(1/10\leq s\leq21/200\),
\[
  \frac1{25}-\frac75s\leq-\frac1{10},
  \qquad
  s^2\leq\left(\frac{21}{200}\right)^2,
\]
and
\[
  (1-s^3)^2
  \geq
  \left(1-\left(\frac{21}{200}\right)^3\right)^2.
\]
Therefore
\[
  \begin{split}
  J(s)
  &\leq
  \left(\frac{21}{200}\right)^2
  -\frac1{90}
   \left(1-\left(\frac{21}{200}\right)^3\right)^2\\
  &=
  -\frac{347909766121}{5760000000000000}
  <0.
  \end{split}
\]
This proves the fourth assertion.
\end{proof}

\subsection{ Construction of Strictly Convex Regions}\label{sec:closing}

\begin{proposition}[Strict convexity of the central section]
\label{prop:counter-central-section}
Let
\[
  \Sigma_0
  =
  \{(0,s,z):V(0,s,z)=0\}.
\]
Then \(\Sigma_0\) is a smooth compact two-dimensional section.  The
full zero-level hypersurface \(\{V=0\}\) is smooth near \(\Sigma_0\),
and at every point of \(\Sigma_0\) its second fundamental form is
positive definite for the normal pointing toward \(\{V>0\}\).
\end{proposition}

\begin{proof}
By \eqref{eq:counter-G0}, \(G(0,s)>0\) exactly for
\(a_0<s<b_0\).  On this interval,
\eqref{eq:counter-MXX} and
Lemma~\ref{lem:counter-signs} give
\[
  \mathcal M_{XX}(0,s)<0.
\]
Equation~\eqref{eq:counter-detM}, together with
\(P_0<0\) and \(Q>0\), gives
\[
  \det\mathcal M(0,s)>0.
\]
Thus the symmetric \(2\times2\) matrix \(\mathcal M\) is negative
definite, and hence
\[
  D^2_{(X,s)}\mathcal R(0,s)<0.
\]

At a point of the zero level with \(z\neq0\), use the defining
function
\[
  F(X,s,z)=|z|-\mathcal R(X,s).
\]
Put \(e=z/|z|\).  If
\((\zeta,\pi)\in\mathbb R^2_{(X,s)}\times\mathbb R^2_z\)
is tangent to \(\{F=0\}\), then
\[
  e\cdot\pi=\nabla\mathcal R\cdot\zeta,
\]
and
\begin{equation}\label{eq:counter-II-radius}
  D^2F[(\zeta,\pi),(\zeta,\pi)]
  =
  \frac{|\pi|^2-(e\cdot\pi)^2}{\mathcal R}
  -
  \zeta^TD^2\mathcal R\,\zeta
  >0.
\end{equation}

It remains to consider the two tips \(z=0\), where
\(s\in\{a_0,b_0\}\).  There,
\[
  V_X=1,
  \qquad
  V_s=\frac{s^3-1}{3},
  \qquad
  V_{Xs}=0,
\]
\[
  V_{ss}=s^2,
  \qquad
  V_{XX}=\frac1{25}-\frac75s.
\]
The tangency condition is
\[
  \zeta_X+\frac{s^3-1}{3}\zeta_s=0.
\]
Consequently,
\[
  D^2V[(\zeta_X,\zeta_s,\pi),
       (\zeta_X,\zeta_s,\pi)]
  =
  J(s)\zeta_s^2+\frac1s|\pi|^2>0
\]
by Lemma~\ref{lem:counter-signs}.  At a tip \(V_s\neq0\); away from
the tips, \(z\neq0\) and
\[
  D_zV=\frac{z}{s}\neq0.
\]
Thus both the section in \(\{X=0\}\) and the full zero-level
hypersurface are smooth at the stated points.  Compactness follows
from
\[
  s\in[a_0,b_0],
  \qquad
  |z|^2=G(0,s).
\]
\end{proof}

\begin{lemma}[A strictly convex zero-level band]
\label{lem:counter-zero-band}
There exists \(\eta_0>0\), which may be chosen so that
\[
  4\eta_0\leq\frac1{10},
\]
such that
\[
  \widehat\Sigma
  =
  \{(X,s,z):V(X,s,z)=0,\ |X|<4\eta_0\}
\]
is a smooth strictly convex hypersurface band.  For every
\(|X|<4\eta_0\), the equation
\[
  \Phi(X,s)=0
\]
has exactly two simple positive roots
\[
  a(X)<b(X),
\]
where \(a\) is strictly convex and \(b\) is strictly concave.
Moreover,
\[
  D_*
  =
  \{(X,s):|X|<4\eta_0,\ a(X)<s<b(X)\}
\]
is convex, and
\[
  \mathcal R(X,s)
  =
  \sqrt{-2s\Phi(X,s)}
\]
is strictly concave on \(D_*\).
\end{lemma}

\begin{proof}
Proposition~\ref{prop:counter-central-section} gives strict convexity
of the full zero-level hypersurface along the compact section
\(\Sigma_0\).  Choose numbers
\[
  0<\underline s<a_0<b_0<\overline s
\]
such that
\[
  \Phi(0,s)>0
  \qquad
  \text{for }
  s\in(0,\underline s]\cup[\overline s,\infty).
\]
The growth of \(s^4/12\), together with continuous dependence of the
coefficients on \(X\), shows that, after shrinking \(\eta_0\), this
positivity remains uniform for \(|X|<4\eta_0\).  Consequently all
zero-level points in this \(X\)-slab lie in a fixed compact cylinder.
If there were zero-level points \(q_j=(X_j,s_j,z_j)\) with
\(X_j\to0\) that stayed outside every fixed neighborhood of
\(\Sigma_0\), a subsequence would converge in that cylinder to a
point \(q_\infty\) with \(X=0\) and \(V(q_\infty)=0\), hence to a
point of \(\Sigma_0\), a contradiction.  Thus all such zero-level
points lie uniformly close to \(\Sigma_0\) after another shrinking of
\(\eta_0\).

Continuity of the gradient and of the tangential quadratic form now
implies that the zero level is smooth and has uniformly positive
tangential Hessian for \(|X|<4\eta_0\).

At \(X=0\), the two roots are simple because
\[
  \Phi_s(0,s)=\frac{s^3-1}{3}\neq0
  \qquad(s=a_0,b_0).
\]
The implicit function theorem therefore gives smooth branches
\(a(X)\) and \(b(X)\) on a neighborhood of
\([-4\eta_0,4\eta_0]\), after shrinking \(\eta_0\) if needed.
Furthermore,
\[
  \Phi_s(X,s)=\frac{s^3}{3}+L(X)
\]
has exactly one positive zero for every sufficiently small \(X\).
Hence \(\Phi(X,\cdot)\) has a unique positive minimum, and the two
simple roots already obtained are all of its positive roots.

For either root branch \(s=s(X)\),
\[
  s'=-\frac{\Phi_X}{\Phi_s}.
\]
At \(X=0\), since \(\Phi_{Xs}=0\) and \(\Phi_X=1\),
\[
  s''(0)
  =
  -\frac{\Phi_{XX}+\Phi_{ss}(s')^2}{\Phi_s}
  =
  -\frac{J(s)}{\Phi_s(0,s)^3}.
\]
At the lower root \(a_0<1\), one has \(\Phi_s<0\), while at the
upper root \(b_0>1\), one has \(\Phi_s>0\).  Therefore
Lemma~\ref{lem:counter-signs} gives
\[
  a''(0)>0,
  \qquad
  b''(0)<0.
\]
Continuity and another shrinking of \(\eta_0\) yield
\[
  a''>0,
  \qquad
  b''<0
\]
throughout the interval.

Thus \(D_*\) is convex: it is the intersection of the epigraph of the
convex function \(a\) and the hypograph of the concave function \(b\).
At points of the zero level with \(z\neq0\), fix \(e=z/|z|\) and, for a given nonzero \(\zeta\in\mathbb R^2\), choose
\[
  \pi=(\nabla\mathcal R\cdot\zeta)e.
\]
Then \((\zeta,\pi)\) is tangent and the first term in
\eqref{eq:counter-II-radius} vanishes. Positivity of the second
fundamental form therefore gives
\[
  -\zeta^TD^2\mathcal R\,\zeta>0.
\]
Hence
\[
  D^2\mathcal R<0
  \qquad\text{on }D_*.
\]
Finally, shrink \(\eta_0\) once more so that
\(4\eta_0\leq1/10\).
\end{proof}

Define
\begin{equation}\label{eq:counter-Kstar}
  K_*
  =
  \left\{
    (X,s,z):
    (X,s)\in\overline{D_*},
    \ |z|\leq\mathcal R(X,s)
  \right\},
\end{equation}
where \(\mathcal R\) is continuously extended to \(\overline{D_*}\); it vanishes on the two boundary graphs \(s=a(X)\) and \(s=b(X)\).

\begin{lemma}[Convexity and boundary structure of \(K_*\)]
\label{lem:counter-Kstar}
The set \(K_*\) is a compact convex body with nonempty interior, and
\begin{equation}\label{eq:counter-Kstar-boundary}
  \partial K_*\cap\{|X|<4\eta_0\}
  =
  \{V=0,\ |X|<4\eta_0\}.
\end{equation}
The right-hand side is a smooth strictly convex hypersurface band.
The boundary \(\partial K_*\) is generally not globally smooth: at
\(X=\pm4\eta_0\) it also contains artificial truncation faces, which
meet the zero-level band along edges.
\end{lemma}

\begin{proof}
Compactness follows from compactness of \(\overline{D_*}\) and
continuity of \(\mathcal R\).  Since \(\mathcal R>0\) on \(D_*\), the
interior is nonempty.

Let \((w_i,z_i)\in K_*\), where \(w_i=(X_i,s_i)\), and let
\(0\leq\theta\leq1\).  By concavity of \(\mathcal R\) and convexity
of \(\overline{D_*}\),
\[
  \begin{split}
  |\theta z_1+(1-\theta)z_2|
  &\leq
  \theta|z_1|+(1-\theta)|z_2|\\
  &\leq
  \theta\mathcal R(w_1)+(1-\theta)\mathcal R(w_2)\\
  &\leq
  \mathcal R(\theta w_1+(1-\theta)w_2).
  \end{split}
\]
Hence \(K_*\) is convex.

For \(|X|<4\eta_0\),
\[
  V\leq0
  \quad\Longleftrightarrow\quad
  |z|^2\leq-2s\Phi(X,s)
  \quad\Longleftrightarrow\quad
  |z|\leq\mathcal R(X,s),
\]
which proves \eqref{eq:counter-Kstar-boundary}.  The artificial end
faces arise directly from the closed truncation in the \(X\)-variable.
\end{proof}

The desired completion is also a consequence of Ghomi's convex
completion theorem \cite[Theorem~1.1.1]{Ghomi}.  We give a direct
construction adapted to the present band.  The artificial end faces
will be removed by two large balls, and the resulting corners will be
rounded by a regularized minimum.  The support-distance branch is used
only in a fixed tubular neighborhood of the smooth central band, where
it is smooth and uniformly strictly concave.

\begin{lemma}[Uniformly convex closing]\label{lem:counter-closing}
There exists a smooth bounded uniformly convex domain
\[
  \widetilde\Omega
  \subset
  \mathbb R_X\times\mathbb R_y^3
\]
such that
\begin{equation}\label{eq:counter-closing-equality}
  \overline{\widetilde\Omega}\cap\{|X|\leq3\eta_0\}
  =
  K_*\cap\{|X|\leq3\eta_0\}.
\end{equation}
In fact, the two closed sets \(\overline{\widetilde\Omega}\) and \(K_*\) agree in a neighborhood of
\[
  K_*\cap\{|X|\leq3\eta_0\}.
\]
Consequently, in a neighborhood of
\[
  \mathcal S
  =
  \{V=0,\ |X|\leq3\eta_0\},
\]
the boundary \(\partial\widetilde\Omega\) agrees exactly with
\(\{V=0\}\).  In particular, for
\[
  I=(-2\eta_0,2\eta_0),
\]
one has
\begin{equation}\label{eq:counter-central-channel}
  \widetilde\Omega\cap(I\times\mathbb R^3)
  =
  \{(X,y):X\in I,\ V(X,y)<0\}.
\end{equation}
\end{lemma}

\begin{proof}
Fix a number
\[
  3\eta_0<\chi<4\eta_0.
\]
Let \(y_0=(1,0,0)\).  By \eqref{eq:counter-profile},
\[
  V(X,y_0)
  =
  -\frac6{25}+X-\frac{17}{25}X^2.
\]
Since \(4\eta_0\leq1/10\), for \(|X|\leq4\eta_0\),
\[
  V(X,y_0)
  \leq
  -\frac6{25}+\frac1{10}
  =
  -\frac7{50}<0.
\]
Thus
\[
  q_0=(0,y_0)
\]
is an interior point of the central model region.

\smallskip
\noindent
\emph{Step 1: a smooth strictly concave defining function near the
central band.}
Let
\[
  h(\omega)
  =
  \max_{p\in K_*}\omega\cdot p,
  \qquad
  |\omega|=1,
\]
be the support function of \(K_*\), and define
\begin{equation}\label{eq:counter-support-distance}
  d(q)
  =
  \min_{|\omega|=1}
  \{h(\omega)-\omega\cdot q\}.
\end{equation}
As an infimum of affine functions, \(d\) is concave.  The supporting
half-space representation of a convex body gives
\[
  K_*=\{d\geq0\},
  \qquad
  \partial K_*=\{d=0\}.
\]

Set
\[
  \Gamma_\chi
  =
  \partial K_*\cap\{|X|\leq\chi\}.
\]
Since \(\chi<4\eta_0\), this compact set is contained in the interior
of the smooth strictly convex band furnished by
Lemma~\ref{lem:counter-zero-band}.  At each \(p\in\Gamma_\chi\), the
supporting direction is the outward unit normal \(\nu(p)\), and its
support hyperplane meets \(K_*\) only at \(p\).  Indeed, if a second
point \(p'\neq p\) lay in the same support hyperplane, convexity would
place the segment \([p,p']\) in \(\partial K_*\).  This would give a
nontrivial boundary segment issuing from \(p\), contradicting the
positive definite second fundamental form there.

We record the localization consequence carefully.  On a slightly
larger compact subband, the Gauss map is a local diffeomorphism because
the second fundamental form is positive definite.  It is also
injective on the set of supporting normals relevant to
\(\Gamma_\chi\): two points with the same normal would lie in the same
support hyperplane, contradicting the preceding uniqueness.  Hence the
normal exponential map
\[
  (p,t)\longmapsto p-t\nu(p)
\]
is a diffeomorphism for \(|t|\) smaller than a uniform positive
constant, after restricting to a neighborhood of \(\Gamma_\chi\).

Equivalently, there is an open tube \(U\) around \(\Gamma_\chi\) in
which the minimizer in \eqref{eq:counter-support-distance} is unique.
For completeness, minimizers coming from the artificial faces cannot
enter this tube.  Otherwise there would be points \(q_j\to p\in
\Gamma_\chi\) and minimizing directions \(\omega_j\) not belonging to
the local normal image.  Passing to a subsequence,
\(\omega_j\to\omega_\infty\).  Since \(d\) is Lipschitz and
\(d(p)=0\),
\[
  h(\omega_\infty)-\omega_\infty\cdot p=0,
\]
so \(\omega_\infty\) supports \(K_*\) at \(p\).  Smoothness at \(p\)
forces \(\omega_\infty=\nu(p)\), and the local normal
diffeomorphism then forces \(\omega_j\) into the local normal image
for large \(j\), a contradiction.

It follows that, on \(U\), \(d\) is precisely the signed Euclidean
distance to \(\partial K_*\), taken positive inside.  In particular,
for \(p\in\Gamma_\chi\) and
\(\tau\in T_p\partial K_*\),
\begin{equation}\label{eq:counter-distance-hessian}
  \nabla d(p)=-\nu(p),
  \qquad
  D^2d(p)[\tau,\tau]
  =
  -\operatorname{II}_p[\tau,\tau],
  \qquad
  D^2d(p)[\nu(p),\cdot]=0.
\end{equation}
No smoothness of \(d\) near the artificial faces or their edges is
being asserted.

Define
\begin{equation}\label{eq:counter-rho}
  \rho=1-e^{-d}.
\end{equation}
The scalar function \(t\mapsto1-e^{-t}\) is increasing and concave;
hence \(\rho\) is concave and has the same sign as \(d\).  On
\(\{d=0\}\),
\[
  D^2\rho
  =
  D^2d-\nabla d\otimes\nabla d.
\]
If \(\xi=\tau+a\nu(p)\), then
\eqref{eq:counter-distance-hessian} yields
\begin{equation}\label{eq:counter-rho-boundary-hessian}
  D^2\rho(p)[\xi,\xi]
  =
  -\operatorname{II}_p[\tau,\tau]-a^2<0
  \qquad(\xi\neq0).
\end{equation}
By compactness of \(\Gamma_\chi\), after shrinking \(U\) there is
\(\lambda>0\) such that
\begin{equation}\label{eq:counter-rho-strong-concavity}
  D^2\rho\leq-\lambda I_4
  \qquad\text{on }U.
\end{equation}

\smallskip
\noindent
\emph{Step 2: two large balls remove the artificial faces.}
For \(\Lambda>0\), define
\begin{align}
  g_+(X,y)
  &=
  (\Lambda+\chi)^2-(X+\Lambda)^2-|y-y_0|^2,
  \label{eq:counter-gplus}\\
  g_-(X,y)
  &=
  (\Lambda+\chi)^2-(X-\Lambda)^2-|y-y_0|^2.
  \label{eq:counter-gminus}
\end{align}
Both functions are strictly concave:
\begin{equation}\label{eq:counter-ball-hessian}
  D^2g_+=D^2g_-=-2I_4.
\end{equation}
The ball \(\{g_+\geq0\}\) lies in \(\{X\leq\chi\}\), whereas
\(\{g_-\geq0\}\) lies in \(\{X\geq-\chi\}\).

Let
\[
  E=K_*\cap\{|X|\leq3\eta_0\}.
\]
This set is compact.  On \(E\),
\[
  g_+(X,y)
  =
  2\Lambda(\chi-X)+\chi^2-X^2-|y-y_0|^2,
\]
\[
  g_-(X,y)
  =
  2\Lambda(\chi+X)+\chi^2-X^2-|y-y_0|^2.
\]
Set
\[
  B_E
  =
  \max_E\bigl(X^2+|y-y_0|^2-\chi^2\bigr)<\infty.
\]
Choose \(\Lambda\) so large that
\[
  2\Lambda(\chi-3\eta_0)\geq B_E+3.
\]
Then
\begin{equation}\label{eq:counter-ball-gap}
  g_+\geq3,
  \qquad
  g_-\geq3
  \qquad\text{on }E.
\end{equation}

\smallskip
\noindent
\emph{Step 3: regularized minima and exact protection of the central
patch.}
Choose an even smooth convex function \(\vartheta_\mu\) such that
\[
  |t|\leq\vartheta_\mu(t)\leq|t|+\mu,
  \qquad
  \vartheta_\mu(t)=|t|\ \text{if }|t|\geq\mu,
  \qquad
  |\vartheta_\mu'(t)|\leq1.
\]
For two functions \(a,b\), put
\begin{equation}\label{eq:counter-regularized-min}
  a\wedge_\mu b
  =
  \frac{a+b-\vartheta_\mu(a-b)}2.
\end{equation}
Then
\begin{equation}\label{eq:counter-regularized-min-bounds}
  \min(a,b)-\frac\mu2
  \leq
  a\wedge_\mu b
  \leq
  \min(a,b).
\end{equation}
Moreover, \(a\wedge_\mu b\) equals \(\min(a,b)\) wherever
\(|a-b|\geq\mu\).  It is nondecreasing in each argument and concave
as a function of the pair \((a,b)\).  Hence the regularized minimum of
concave functions is concave.

Set
\begin{equation}\label{eq:counter-rhomu}
  \rho_\mu
  =
  (\rho\wedge_\mu g_+)\wedge_\mu g_-,
  \qquad
  \widetilde K=\{\rho_\mu\geq0\}.
\end{equation}
The function \(\rho_\mu\) is concave, so \(\widetilde K\) is convex.
Also,
\[
  \rho_\mu\leq\rho,
  \qquad
  \rho_\mu\leq g_+,
  \qquad
  \rho_\mu\leq g_-,
\]
and therefore
\[
  \widetilde K
  \subset
  K_*\cap\{g_+\geq0\}\cap\{g_-\geq0\}
  \subset
  K_*\cap\{|X|\leq\chi\}.
\]
Thus \(\widetilde K\) is compact.

Inside \(K_*\), one has \(0\leq\rho<1\).  In view of
\eqref{eq:counter-ball-gap}, if \(0<\mu<1\), then on \(E\)
\[
  g_\pm-\rho>2>\mu.
\]
Both regularized minima therefore select \(\rho\) exactly:
\begin{equation}\label{eq:counter-protected-branch}
  \rho_\mu=\rho
  \qquad\text{on }E.
\end{equation}
This is a strict protected-branch gap.  Since \(E\) is compact and
all the functions involved are continuous, the same branch selection
holds on an open neighborhood \(N_E\) of \(E\):
\begin{equation}\label{eq:counter-protected-neighborhood}
  \rho_\mu=\rho
  \qquad\text{on }N_E.
\end{equation}
If \(q\notin K_*\), then \(\rho(q)<0\) and
\(\rho_\mu(q)\leq\rho(q)\), so \(q\notin\widetilde K\).  Together
with \eqref{eq:counter-protected-branch}, this proves
\begin{equation}\label{eq:counter-Ktilde-central}
  \widetilde K\cap\{|X|\leq3\eta_0\}
  =
  K_*\cap\{|X|\leq3\eta_0\},
\end{equation}
and \eqref{eq:counter-protected-neighborhood} gives equality in a
neighborhood of the protected central set.  In particular \(q_0\) is
an interior point of \(\widetilde K\).

\smallskip
\noindent
\emph{Step 4: active-branch localization and smooth strict
concavity near the new boundary.}
For smooth \(a,b\), differentiation of
\eqref{eq:counter-regularized-min} gives
\begin{equation}\label{eq:counter-min-hessian}
  \begin{split}
  D^2(a\wedge_\mu b)
  ={}&
  \alpha D^2a+(1-\alpha)D^2b\\
  &-
  \frac{\vartheta_\mu''(a-b)}2
  (Da-Db)\otimes(Da-Db),
  \end{split}
\end{equation}
where \(0\leq\alpha\leq1\).  The last term is negative
semidefinite.

Let \(q\in\partial\widetilde K\), and put
\[
  h=\rho\wedge_\mu g_+.
\]
Since
\[
  0=(h\wedge_\mu g_-)(q)
\]
and the regularized minimum is no larger than either input,
\[
  h(q),\ g_-(q),\ \rho(q),\ g_+(q)\geq0.
\]
By \eqref{eq:counter-regularized-min-bounds},
\[
  \min\{h(q),g_-(q)\}\leq\frac\mu2.
\]
Suppose first that the outer regularized minimum depends on \(h\) in
every neighborhood of \(q\).  Then continuity implies
\[
  h(q)-g_-(q)\leq\mu;
\]
otherwise it would equal \(g_-\) and be independent of \(h\) in a
neighborhood of \(q\).  If \(h(q)\leq g_-(q)\), then
\(h(q)=\min\{h(q),g_-(q)\}\leq\mu/2\).  If
\(g_-(q)<h(q)\), then \(g_-(q)\leq\mu/2\) and
\(h(q)\leq g_-(q)+\mu\).  Thus, including the threshold case
\(|h-g_-|=\mu\),
\begin{equation}\label{eq:counter-h-small}
  h(q)\leq\frac{3\mu}{2}.
\end{equation}
Applying \eqref{eq:counter-regularized-min-bounds} to
\(h=\rho\wedge_\mu g_+\), we find
\[
  \min\{\rho(q),g_+(q)\}\leq2\mu.
\]
If the inner regularized minimum depends on \(\rho\) in every
neighborhood of \(q\), the same argument gives
\(\rho(q)-g_+(q)\leq\mu\).  If \(\rho(q)\leq g_+(q)\), then
\(\rho(q)\leq2\mu\); otherwise
\(g_+(q)\leq2\mu\) and \(\rho(q)\leq g_+(q)+\mu\).  Hence
\begin{equation}\label{eq:counter-rho-small}
  0\leq\rho(q)\leq3\mu.
\end{equation}

Now the compact set
\[
  \bigl(K_*\cap\{|X|\leq\chi\}\bigr)\setminus U
\]
does not meet \(\partial K_*\). If it is nonempty, \(\rho\) has a strictly positive
minimum there; denote it by \(m_U\). If it is empty, fix any
\(m_U>0\). Choose \(\mu\) so small that
\[
  0<\mu<1,
  \qquad
  3\mu<m_U.
\]
Then
\begin{equation}\label{eq:counter-active-localization}
  \{0\leq\rho\leq3\mu,\ |X|\leq\chi\}\subset U.
\end{equation}

We can now inspect all possible active branches at \(q\).
\begin{itemize}
\item
If there is a neighborhood of \(q\) on which the outer minimum is
independent of \(h\), then it equals \(g_-\) there.
\item
Otherwise the estimate \eqref{eq:counter-h-small} applies.  If there
is a neighborhood of \(q\) on which the inner minimum is independent
of \(\rho\), then the inner minimum equals \(g_+\) there, so the
outer minimum involves only the two smooth ball branches.
\item
If neither of the preceding neighborhoods exists, then the estimate
\eqref{eq:counter-rho-small} applies, and
\eqref{eq:counter-active-localization} puts \(q\) in \(U\), where
\(\rho\) is smooth and satisfies
\eqref{eq:counter-rho-strong-concavity}.
\end{itemize}
Thus every branch active near the new boundary is smooth there.  The
ball branches have Hessian \(-2I_4\), and the \(\rho\)-branch, whenever
active, has Hessian at most \(-\lambda I_4\).  Formula
\eqref{eq:counter-min-hessian} therefore shows, point by point on
\(\partial\widetilde K\), that \(\rho_\mu\) is smooth and has negative
definite Hessian in a neighborhood of that point.  By compactness of
the new boundary, there is a single neighborhood on which
\begin{equation}\label{eq:counter-rhomu-strict}
  D^2\rho_\mu<0.
\end{equation}

By \eqref{eq:counter-protected-branch},
\(\rho_\mu(q_0)>0\). Concavity first shows that
\begin{equation}\label{eq:counter-positive-interior}
  \operatorname{int}\widetilde K=\{\rho_\mu>0\}.
\end{equation}
Indeed, if an interior point \(q\) satisfied \(\rho_\mu(q)=0\), then for sufficiently small \(\delta>0\) the point
\[
  q'=q+\delta(q-q_0)
\]
would still lie in \(\widetilde K\), and
\[
  q=\frac1{1+\delta}q'
  +\frac{\delta}{1+\delta}q_0.
\]
Concavity would then give \(\rho_\mu(q)>0\), a contradiction. Thus every zero of \(\rho_\mu\) lies on \(\partial\widetilde K\), where the preceding argument shows that \(\rho_\mu\) is smooth.

It remains to prove that zero is a regular value. If \(\rho_\mu(q)=0\), concavity gives
\[
  0<\rho_\mu(q_0)
  \leq
  \rho_\mu(q)+
  \nabla\rho_\mu(q)\cdot(q_0-q)
  =
  \nabla\rho_\mu(q)\cdot(q_0-q).
\]
Hence \(\nabla\rho_\mu(q)\neq0\).  Define
\[
  \widetilde\Omega
  =
  \operatorname{int}\widetilde K
  =
  \{\rho_\mu>0\},
\]
where the last identity is \eqref{eq:counter-positive-interior}.
Its outward normal is
\[
  -\frac{\nabla\rho_\mu}{|\nabla\rho_\mu|}.
\]
For every nonzero tangent vector \(\tau\),
\[
  \operatorname{II}_{\partial\widetilde\Omega}(\tau,\tau)
  =
  -\frac{D^2\rho_\mu[\tau,\tau]}
         {|\nabla\rho_\mu|}
  >0.
\]
The boundary is compact, so its principal curvatures have a uniform
positive lower bound.  Thus \(\widetilde\Omega\) is uniformly convex.
Since \(\widetilde K\) is a convex body with nonempty interior,
\(\overline{\widetilde\Omega}=\widetilde K\). Therefore
Equation~\eqref{eq:counter-closing-equality} follows from
\eqref{eq:counter-Ktilde-central}; the neighborhood assertion follows
from \eqref{eq:counter-protected-neighborhood}; and
\eqref{eq:counter-central-channel} follows immediately from
\eqref{eq:counter-Kstar-boundary}.
\end{proof}

\subsection{Approximation in the slow channel}
\label{sec:channel-comparison}

For \(0<\varepsilon<1\), define the physical domain
\begin{equation}\label{eq:counter-stretched-domain}
  \Omega_\varepsilon
  =
  \{(x,y)\in\mathbb R\times\mathbb R^3:
    (\varepsilon x,y)\in\widetilde\Omega\}.
\end{equation}
It is obtained from \(\widetilde\Omega\) by an invertible linear
stretch in the \(X\)-direction.  Hence, for every fixed
\(\varepsilon>0\), the domain \(\Omega_\varepsilon\) is smooth,
bounded, and uniformly convex.  Classical Hessian Dirichlet theory
\cite{caffarelli1985dirichlet} gives a unique
\[
  u_\varepsilon\in C^\infty(\overline{\Omega_\varepsilon})
\]
satisfying
\begin{equation}\label{eq:counter-physical-problem}
  \begin{cases}
  \sigma_3(D^2u_\varepsilon)=1
    &\text{in }\Omega_\varepsilon,\\
  u_\varepsilon=0
    &\text{on }\partial\Omega_\varepsilon,\\
  D^2u_\varepsilon\in\Gamma_3
    &\text{in }\overline{\Omega_\varepsilon}.
  \end{cases}
\end{equation}

Define the pullback to slow coordinates by
\begin{equation}\label{eq:counter-Ueps}
  U_\varepsilon(X,y)
  =
  u_\varepsilon(X/\varepsilon,y),
  \qquad
  (X,y)\in\widetilde\Omega.
\end{equation}

\begin{lemma}[An \(\varepsilon\)-independent bound]
\label{lem:counter-coarse-bound}
There is a constant \(C_0>0\), independent of \(\varepsilon\), such
that
\begin{equation}\label{eq:counter-coarse-bound}
  -C_0\leq u_\varepsilon\leq0
  \qquad\text{in }\Omega_\varepsilon.
\end{equation}
\end{lemma}

\begin{proof}
Since \(D^2u_\varepsilon\in\Gamma_3\),
\[
  \Delta u_\varepsilon
  =
  \sigma_1(D^2u_\varepsilon)>0.
\]
The maximum principle gives \(u_\varepsilon\leq0\).

Choose \(R_*>0\) so that the \(y\)-projection of
\(\widetilde\Omega\) is contained in \(B_{R_*}^3\).  In physical
variables, set
\[
  \psi(x,y)
  =
  \frac12(|y|^2-R_*^2).
\]
Its Hessian has eigenvalues \((0,1,1,1)\).  In particular,
\[
  \sigma_1(D^2\psi)=3,
  \qquad
  \sigma_2(D^2\psi)=3,
  \qquad
  \sigma_3(D^2\psi)=1,
\]
so \(D^2\psi\in\Gamma_3\).  On \(\partial\Omega_\varepsilon\),
\[
  \psi\leq0=u_\varepsilon.
\]
The \(\Gamma_3\) comparison principle gives
\(\psi\leq u_\varepsilon\).  Thus one may take
\[
  C_0=\frac{R_*^2}{2}.
\]
\end{proof}

For a symmetric \(3\times3\) matrix \(A\), let
\[
  T_1(A)=\sigma_1(A)I_3-A
\]
be its first Newton transformation.

\begin{lemma}[\(3+1\) block formula]\label{lem:counter-block}
If \(A\in\operatorname{Sym}(3)\), \(\beta\in\mathbb R^3\), and
\(d\in\mathbb R\), then
\begin{equation}\label{eq:counter-block}
  \sigma_3
  \begin{pmatrix}
    A&\beta\\
    \beta^T&d
  \end{pmatrix}
  =
  \det A+d\sigma_2(A)-\beta^TT_1(A)\beta.
\end{equation}
\end{lemma}

\begin{proof}
The quantity \(\sigma_3\) is the sum of the four \(3\times3\)
principal minors.  The minor obtained by deleting the last row and
column is \(\det A\).  In the other three minors, the terms containing
\(d\) add up to \(d\sigma_2(A)\), while the terms quadratic in
\(\beta\) add up to
\(-\beta^TT_1(A)\beta\).  Expanding these three minors proves
\eqref{eq:counter-block}.
\end{proof}

For a slow-coordinate function \(W=W(X,y)\), define its physical
Hessian by
\begin{equation}\label{eq:counter-Heps}
  \mathcal H_\varepsilon[W]
  =
  \begin{pmatrix}
    D_y^2W&\varepsilon D_yW_X\\
    \varepsilon(D_yW_X)^T&\varepsilon^2W_{XX}
  \end{pmatrix}.
\end{equation}
The change of variables \(X=\varepsilon x\) shows that
\begin{equation}\label{eq:counter-Ueq}
  \sigma_3(\mathcal H_\varepsilon[U_\varepsilon])=1.
\end{equation}

For the model profile, write
\[
  A=D_y^2V,
  \qquad
  b=D_yV_X,
  \qquad
  c=V_{XX}.
\]
From \eqref{eq:counter-CL},
\begin{equation}\label{eq:counter-bc}
  b=
  \begin{pmatrix}
    -\dfrac75X\\0\\0
  \end{pmatrix},
  \qquad
  c=\frac1{25}-\frac75s.
\end{equation}
Using \(\det A=1\) and Lemma~\ref{lem:counter-block},
\begin{equation}\label{eq:counter-residual}
  \sigma_3(\mathcal H_\varepsilon[V])
  =
  1+\varepsilon^2\mathcal E,
  \qquad
  \mathcal E
  =
  c\sigma_2(A)-b^TT_1(A)b.
\end{equation}

\begin{lemma}[Exact quadratic-barrier identity]
\label{lem:counter-barrier-identity}
Let \(\theta>0\), \(A_0\in\mathbb R\), and \(X_0\in\mathbb R\), and
set
\[
  W=\theta V+A_0(X-X_0)^2.
\]
Then
\begin{equation}\label{eq:counter-barrier-identity}
  \sigma_3(\mathcal H_\varepsilon[W])
  =
  \theta^3(1+\varepsilon^2\mathcal E)
  +
  2A_0\theta^2\varepsilon^2\sigma_2(A).
\end{equation}
\end{lemma}

\begin{proof}
We have
\[
  D_y^2W=\theta A,
  \qquad
  D_yW_X=\theta b,
  \qquad
  W_{XX}=\theta c+2A_0.
\]
Substitution into \eqref{eq:counter-block}, together with
\[
  \det(\theta A)=\theta^3\det A,
  \qquad
  \sigma_2(\theta A)=\theta^2\sigma_2(A),
  \qquad
  T_1(\theta A)=\theta T_1(A),
\]
gives \eqref{eq:counter-barrier-identity}.
\end{proof}

\begin{lemma}[\(\Gamma_3\) comparison principle]
\label{lem:counter-comparison}
Let \(D\) be bounded, and suppose that
\[
  u,v\in C^2(\overline D),
  \qquad
  D^2u(x),D^2v(x)\in\Gamma_3
  \quad\text{for every }x\in\overline D.
\]
If
\[
  \sigma_3(D^2u)\geq\sigma_3(D^2v)
  \quad\text{in }D,
  \qquad
  u\leq v
  \quad\text{on }\partial D,
\]
then \(u\leq v\) in \(D\).
\end{lemma}

\begin{proof}
The G{\aa}rding cone \(\Gamma_3\) is convex.  Put \(w=u-v\).
The differential identity
\[
  D\sigma_3(M)=T_2(M)
\]
gives
\[
  0
  \leq
  \sigma_3(D^2u)-\sigma_3(D^2v)
  =
  a^{ij}w_{ij},
\]
where
\[
  a^{ij}
  =
  \int_0^1
  T_2\bigl(D^2v+t(D^2u-D^2v)\bigr)_{ij}\,dt.
\]
For matrices in \(\Gamma_3\), the second Newton transformation
\(T_2\) is positive definite.  Since the segment of Hessians is a
compact subset of \(\Gamma_3\), the matrix \(a^{ij}\) is uniformly
positive definite.  The linear maximum principle yields \(w\leq0\).
\end{proof}

Choose intervals
\[
  I_1\Subset I_2\Subset I=(-2\eta_0,2\eta_0),
  \qquad
  0\in I_1,
\]
and set
\[
  d_0=\operatorname{dist}(\overline I_1,\partial I_2)>0,
  \qquad
  r=\frac{d_0}{2}.
\]
Define the compact channel
\begin{equation}\label{eq:counter-Q2}
  Q_2
  =
  \{(X,y):X\in\overline I_2,\ V(X,y)\leq0\}.
\end{equation}
By \eqref{eq:counter-central-channel}, this is the genuine central
channel in \(\widetilde\Omega\).  On \(Q_2\), put
\begin{equation}\label{eq:counter-channel-constants}
  M_V=\sup_{Q_2}(-V),
  \qquad
  E_0=\sup_{Q_2}|\mathcal E|,
  \qquad
  S_0=\sup_{Q_2}\sigma_2(A).
\end{equation}
These quantities are finite, and \(S_0>0\).

\begin{theorem}[\(O(\varepsilon^2)\) slow-channel approximation]
\label{thm:counter-channel}
There are constants \(C>0\) and \(\varepsilon_1>0\) such that, for
every \(0<\varepsilon<\varepsilon_1\),
\begin{equation}\label{eq:counter-channel-estimate}
  \sup_{\{X\in\overline I_1,\ V(X,y)\leq0\}}
  |U_\varepsilon(X,y)-V(X,y)|
  \leq C\varepsilon^2.
\end{equation}
\end{theorem}

\begin{proof}
We choose the constants in a noncircular order.  First fix the global
bound \(C_0\) from Lemma~\ref{lem:counter-coarse-bound}, and choose
\(A_0>0\) so that
\begin{equation}\label{eq:counter-A0-choice}
  A_0r^2\geq C_0+M_V+1.
\end{equation}
Next choose \(K_0>0\) sufficiently large that
\begin{equation}\label{eq:counter-K0-choice}
  \frac32K_0
  \geq
  E_0+2A_0S_0+1,
  \qquad
  3K_0
  \geq
  \frac{27}{8}E_0+\frac92A_0S_0+1.
\end{equation}
Only after \(A_0\) and \(K_0\) have been fixed do we shrink
\(\varepsilon_1\) so that, for \(0<\varepsilon<\varepsilon_1\),
\begin{equation}\label{eq:counter-deltaeps}
  \delta_\varepsilon
  :=
  K_0\varepsilon^2
  \leq\frac12,
\end{equation}
and all barriers below are \(\Gamma_3\)-admissible.

The last assertion follows uniformly from compactness.  At
\(\varepsilon=0\), the barrier matrices have the form
\[
  \operatorname{diag}(\theta A,0),
  \qquad
  \frac12\leq\theta\leq\frac32.
\]
Since \(A>0\) and \(\det A=1\), their first three elementary symmetric
functions have uniform positive lower bounds on the compact set
\(Q_2\).  The actual barrier matrices are a uniformly small
\(\varepsilon\)-perturbation of this compact family, and hence remain
in the open cone \(\Gamma_3\) when \(\varepsilon_1\) is sufficiently
small.

Fix \(X_0\in\overline I_1\) and consider the artificial short cylinder
\begin{equation}\label{eq:counter-short-cylinder}
  D_{X_0}
  =
  \{(X,y):|X-X_0|<r,\ V(X,y)<0\}
  \subset Q_2.
\end{equation}
Define
\begin{align}
  \underline W
  &=
  (1+\delta_\varepsilon)V-A_0(X-X_0)^2,
  \label{eq:counter-lower-barrier}\\
  \overline W
  &=
  (1-\delta_\varepsilon)V+A_0(X-X_0)^2.
  \label{eq:counter-upper-barrier}
\end{align}
On the genuine side boundary \(V=0\),
\eqref{eq:counter-central-channel} gives \(U_\varepsilon=0\), and
hence
\[
  \underline W\leq U_\varepsilon\leq\overline W.
\]
On the artificial ends \(|X-X_0|=r\), because \(V\leq0\),
\[
  \underline W
  \leq
  -A_0r^2
  \leq
  -C_0
  \leq
  U_\varepsilon,
\]
while
\[
  \overline W
  \geq
  V+A_0r^2
  \geq
  -M_V+A_0r^2
  \geq
  0
  \geq
  U_\varepsilon.
\]
Thus the boundary ordering holds.

The exact identity
\eqref{eq:counter-barrier-identity} gives
\[
  \sigma_3(\mathcal H_\varepsilon[\overline W])
  =
  (1-\delta_\varepsilon)^3
  (1+\varepsilon^2\mathcal E)
  +
  2A_0(1-\delta_\varepsilon)^2
  \varepsilon^2\sigma_2(A),
\]
and
\[
  \sigma_3(\mathcal H_\varepsilon[\underline W])
  =
  (1+\delta_\varepsilon)^3
  (1+\varepsilon^2\mathcal E)
  -
  2A_0(1+\delta_\varepsilon)^2
  \varepsilon^2\sigma_2(A).
\]
For \(0\leq\delta_\varepsilon\leq1/2\),
\[
  (1-\delta_\varepsilon)^3
  \leq
  1-\frac32\delta_\varepsilon,
  \qquad
  (1+\delta_\varepsilon)^3
  \geq
  1+3\delta_\varepsilon,
\]
and
\[
  (1+\delta_\varepsilon)^3\leq\frac{27}{8},
  \qquad
  (1+\delta_\varepsilon)^2\leq\frac94.
\]
Using \eqref{eq:counter-K0-choice}, we obtain
\[
  \begin{split}
  \sigma_3(\mathcal H_\varepsilon[\overline W])
  &\leq
  1+
  \varepsilon^2
  \left(
    -\frac32K_0+E_0+2A_0S_0
  \right)
  \leq1,
  \end{split}
\]
and
\[
  \begin{split}
  \sigma_3(\mathcal H_\varepsilon[\underline W])
  &\geq
  1+
  \varepsilon^2
  \left(
    3K_0-\frac{27}{8}E_0-\frac92A_0S_0
  \right)
  \geq1.
  \end{split}
\]

Apply the linear change of variables
\[
  (X,y)\longmapsto(X/\varepsilon,y)
\]
to \(D_{X_0}\) and to the three functions.  Under this map,
\(\mathcal H_\varepsilon\) becomes the ordinary physical Hessian.
Lemma~\ref{lem:counter-comparison}, applied on the physical image,
therefore gives
\[
  \underline W\leq U_\varepsilon\leq\overline W
  \qquad\text{in }D_{X_0}.
\]
At \(X=X_0\),
\[
  (1+\delta_\varepsilon)V(X_0,y)
  \leq
  U_\varepsilon(X_0,y)
  \leq
  (1-\delta_\varepsilon)V(X_0,y).
\]
Since \(V\leq0\),
\[
  |U_\varepsilon(X_0,y)-V(X_0,y)|
  \leq
  \delta_\varepsilon(-V(X_0,y))
  \leq
  K_0M_V\varepsilon^2.
\]
As \(X_0\in\overline I_1\) is arbitrary,
\eqref{eq:counter-channel-estimate} follows with
\[
  C=K_0M_V.
\]
\end{proof}

\subsection{A nonconvex interior level}\label{sec:midpoint}

Set
\begin{equation}\label{eq:counter-cstar}
  c_*=-\frac3{125}.
\end{equation}
On \(X=0\) and \(z=0\), the equation \(V=c_*\) is equivalent to
\[
  \frac{s^4}{12}-\frac{s}{3}+\frac1{100}
  =
  -\frac3{125},
\]
which is precisely \eqref{eq:counter-bad-quartic}.  By
Lemma~\ref{lem:counter-signs}, there is a unique
\[
  \frac1{10}<s_*<\frac{21}{200}
\]
solving this equation, and \(J(s_*)<0\).  Let
\[
  q_*=(X,s,z_1,z_2)=(0,s_*,0,0).
\]
At this point,
\[
  V_X(q_*)=1,
  \qquad
  V_s(q_*)=\frac{s_*^3-1}{3}.
\]
Consequently,
\begin{equation}\label{eq:counter-bad-tangent}
  \tau
  =
  \left(
    \frac{s_*^3-1}{3},-1,0,0
  \right)
\end{equation}
is tangent to the level \(\{V=c_*\}\), and
\begin{equation}\label{eq:counter-negative-tangent}
  D^2V(q_*)[\tau,\tau]
  =
  J(s_*)<0.
\end{equation}

\begin{lemma}[A midpoint gap from negative tangential curvature]
\label{lem:counter-midpoint}
Let \(f\in C^3\), let \(f(q)=c\), and suppose that
\(\nabla f(q)\neq0\).  If there is a vector \(\tau\neq0\) such that
\[
  \nabla f(q)\cdot\tau=0,
  \qquad
  D^2f(q)[\tau,\tau]=-\kappa<0,
\]
then every sufficiently small neighborhood of \(q\) contains points
\(q_+,q_-\) and a number \(\gamma>0\) such that
\begin{equation}\label{eq:counter-midpoint-abstract}
  f(q_\pm)\leq c-4\gamma,
  \qquad
  f\left(\frac{q_++q_-}{2}\right)\geq c+4\gamma.
\end{equation}
\end{lemma}

\begin{proof}
Put
\[
  g=|\nabla f(q)|,
  \qquad
  n=\frac{\nabla f(q)}{g}.
\]
By the implicit function theorem, for all sufficiently small \(h\)
there are two points on the same level,
\[
  r_\pm(h)
  =
  q\pm h\tau+\beta_\pm(h)n,
  \qquad
  f(r_\pm(h))=c.
\]
Taylor expansion, using
\(\nabla f(q)\cdot\tau=0\) and
\(D^2f(q)[\tau,\tau]=-\kappa\), yields
\[
  \beta_\pm(h)
  =
  \frac{\kappa}{2g}h^2+O(h^3).
\]
Let
\[
  \rho_0=\frac{\kappa}{8g}
\]
and define
\[
  q_\pm=r_\pm(h)-\rho_0h^2n.
\]
Expanding from the same-level points \(r_\pm(h)\), we obtain
\[
  f(q_\pm)
  =
  c-\frac{\kappa}{8}h^2+O(h^3).
\]
On the other hand,
\[
  \frac{q_++q_-}{2}
  =
  q+
  \left(
    \frac{\kappa}{2g}-\frac{\kappa}{8g}
  \right)h^2n
  +O(h^3)
  =
  q+\frac{3\kappa}{8g}h^2n+O(h^3),
\]
and hence
\[
  f\left(\frac{q_++q_-}{2}\right)
  =
  c+\frac{3\kappa}{8}h^2+O(h^3).
\]
Choose \(h>0\) sufficiently small.  Both displayed margins are then
strictly positive.  Choosing \(\gamma\) smaller than one quarter of
the smaller margin proves \eqref{eq:counter-midpoint-abstract}.
\end{proof}

Apply Lemma~\ref{lem:counter-midpoint} to
\[
  f=V,
  \qquad
  q=q_*,
  \qquad
  c=c_*.
\]
Since \(c_*<0\), the points may be chosen sufficiently close to
\(q_*\) that they and their midpoint lie in the fixed central channel
\[
  \{X\in I_1,\ V<0\}.
\]
Thus there are fixed points \(q_+,q_-\), their midpoint
\[
  m=\frac{q_++q_-}{2},
\]
and a number \(\gamma>0\) such that
\begin{equation}\label{eq:counter-fixed-gap}
  V(q_\pm)\leq c_*-4\gamma,
  \qquad
  V(m)\geq c_*+4\gamma.
\end{equation}
The points and \(\gamma\) are independent of \(\varepsilon\).

\subsection{Completion of the counterexample}\label{sec:counterexample-proof}

\begin{proof}[Proof of Theorem~\ref{thm:main-E}]
The smoothness, boundedness, and uniform convexity of
\(\Omega_\varepsilon\) follow from
Lemma~\ref{lem:counter-closing} and the invertible linear stretch
\eqref{eq:counter-stretched-domain}.  Classical solvability
\cite{caffarelli1985dirichlet} provides the unique smooth admissible solution
\(u_\varepsilon\).

Apply Theorem~\ref{thm:counter-channel} to the three fixed points in
\eqref{eq:counter-fixed-gap}.  For all sufficiently small
\(\varepsilon\),
\[
  \|U_\varepsilon-V\|_{L^\infty}<\gamma
\]
on a fixed compact central channel containing these points.  Therefore
\[
  U_\varepsilon(q_\pm)\leq c_*-3\gamma,
  \qquad
  U_\varepsilon(m)\geq c_*+3\gamma.
\]
It follows that
\[
  q_+,q_-\in\{U_\varepsilon<c_*\},
  \qquad
  \frac{q_++q_-}{2}=m\notin\{U_\varepsilon<c_*\}.
\]
Thus the slow-coordinate sublevel set is nonconvex.

Let
\[
  T_\varepsilon^{-1}(X,y)
  =
  (X/\varepsilon,y)
\]
be the map from slow to physical coordinates, and set
\[
  p_\pm=T_\varepsilon^{-1}(q_\pm),
  \qquad
  p_m=T_\varepsilon^{-1}(m).
\]
This is an invertible real linear map, so it preserves line segments
and midpoints:
\[
  p_m=\frac{p_++p_-}{2}.
\]
Moreover,
\[
  p_\pm\in\{u_\varepsilon<c_*\},
  \qquad
  p_m\notin\{u_\varepsilon<c_*\}.
\]
Hence
\[
  \left\{
    u_\varepsilon<-\frac3{125}
  \right\}
\]
is nonconvex.
\end{proof}

\begin{proof}[Proof of Corollary~\ref{thm:main-F}]
Let $c_*=-3/125$. Since $\Psi$ is strictly increasing on the range of $u$,
\[
\{\Psi(u)<\Psi(c_*)\}=\{u<c_*\}.
\]
The set on the right is nonconvex by Theorem~\ref{thm:main-E}. If $\Psi\circ u$ were convex, all of its sublevel sets would be convex, which is impossible. Taking $\Psi(s)=-(-s)^\alpha$ gives the final assertion for every $\alpha>0$.
\end{proof}

\begin{remark}
\label{rem:counter-scope}
Each fixed domain \(\Omega_\varepsilon\) is smooth, bounded, and
uniformly convex, which is exactly what is needed to disprove an
unconditional extension of Theorem~\ref{thm:mainssss}.  The family is, however, increasingly elongated as \(\varepsilon\downarrow0\), and
its quantitative uniform-convexity and eccentricity constants are not
uniform in \(\varepsilon\).  The example therefore does not rule out
positive convexity results under additional quantitative geometric
hypotheses imposed uniformly on the domain.  Nor does it make a claim
about the principal \(3\)-Hessian eigenfunction or its
Brunn--Minkowski theory.
\end{remark}
\section*{Acknowledgments}

Jiahuan Li and Xi-Nan Ma were supported by the National Key R\&D
Program of China (Grant No.\ 2025YFA1017601).   Guohuan Qiu was supported by the National Natural Science Foundation of China (Grant No. 12571227). Paolo Salani was
partially supported by INdAM through GNAMPA and by the project
\emph{Geometric-Analytic Methods for PDEs and Applications} (GAMPA),
funded by the European Union--NextGenerationEU under the PRIN 2022
program (D.D.\ 104, 02/02/2022, Ministero dell'Universit\`a e della
Ricerca).

\small
\bibliographystyle{plain}
\bibliography{reference}

\end{document}